\batchmode
\makeatletter
\def\input@path{{"/home/andrew/Documents/Andrew Backup October 2022/andrew/"}}
\makeatother
\documentclass[10pt,oneside,british,english]{amsart}
\usepackage[T1]{fontenc}
\usepackage[latin9]{inputenc}
\usepackage{geometry}
\usepackage{verbatim}
\usepackage{amstext}
\usepackage{amsthm}
\usepackage{amssymb}
\usepackage{esint}

\makeatletter
\theoremstyle{plain}
\newtheorem{thm}{\protect\theoremname}
\theoremstyle{definition}
\newtheorem{defn}[thm]{\protect\definitionname}
\theoremstyle{remark}
\newtheorem*{rem*}{\protect\remarkname}
\theoremstyle{plain}
\newtheorem{prop}[thm]{\protect\propositionname}
\theoremstyle{definition}
\newtheorem{example}[thm]{\protect\examplename}
\theoremstyle{plain}
\newtheorem{lem}[thm]{\protect\lemmaname}
\theoremstyle{plain}
\newtheorem{cor}[thm]{\protect\corollaryname}
\theoremstyle{plain}
\newtheorem*{thm*}{\protect\theoremname}

\RequirePackage{amsmath, amssymb}
\RequirePackage[colorlinks,citecolor=blue,urlcolor=blue]{hyperref}

\makeatother

\usepackage{babel}
\addto\captionsbritish{\renewcommand{\corollaryname}{Corollary}}
\addto\captionsbritish{\renewcommand{\definitionname}{Definition}}
\addto\captionsbritish{\renewcommand{\examplename}{Example}}
\addto\captionsbritish{\renewcommand{\lemmaname}{Lemma}}
\addto\captionsbritish{\renewcommand{\propositionname}{Proposition}}
\addto\captionsbritish{\renewcommand{\remarkname}{Remark}}
\addto\captionsbritish{\renewcommand{\theoremname}{Theorem}}
\addto\captionsenglish{\renewcommand{\corollaryname}{Corollary}}
\addto\captionsenglish{\renewcommand{\definitionname}{Definition}}
\addto\captionsenglish{\renewcommand{\examplename}{Example}}
\addto\captionsenglish{\renewcommand{\lemmaname}{Lemma}}
\addto\captionsenglish{\renewcommand{\propositionname}{Proposition}}
\addto\captionsenglish{\renewcommand{\remarkname}{Remark}}
\addto\captionsenglish{\renewcommand{\theoremname}{Theorem}}
\providecommand{\corollaryname}{Corollary}
\providecommand{\definitionname}{Definition}
\providecommand{\examplename}{Example}
\providecommand{\lemmaname}{Lemma}
\providecommand{\propositionname}{Proposition}
\providecommand{\remarkname}{Remark}
\providecommand{\theoremname}{Theorem}

\begin{document}
\title{Wasserstein Conditional Independence Testing}
\author{Andrew Warren}
\address{Carnegie Mellon University\foreignlanguage{british}{\\ 5000 Forbes
Avenue\\ Pittsburgh, PA 15213\\ USA} }
\email{awarren1@alumni.cmu.edu}
\begin{abstract}
We introduce a test for the conditional independence of random variables
$X$ and $Y$ given a random variable $Z$, specifically by sampling
from the joint distribution $(X,Y,Z)$, binning the support of the
distribution of $Z$, and conducting multiple $p$-Wasserstein two-sample
tests. Under a $p$-Wasserstein Lipschitz assumption on the conditional
distributions $\mathcal{L}_{X|Z}$, $\mathcal{L}_{Y|Z}$, and $\mathcal{L}_{(X,Y)|Z}$,
we show that it is possible to control the Type I and Type II error
of this test, and give examples of explicit finite-sample error bounds
in the case where the distribution of $Z$ has compact support. %
\end{abstract}

\maketitle
Consider random variables $X,Y,Z$. That is, $(\Omega,\mathcal{A},\mathbb{P})$
is our sample space, and $X:\Omega\rightarrow\mathbb{R}^{d_{X}}$,
$Y:\Omega\rightarrow\mathbb{R}^{d_{Y}}$, and $Z:\Omega\rightarrow\mathbb{R}^{d_{Z}}$
are all Borel functions, with distributions $\mathcal{L}_{X}:=\mathbb{P}\circ X^{-1}$
(and similar for $Y$ and $Z$) as well as joint distribution $\mathcal{L}_{(X,Y,Z)}:=\mathbb{P}\circ(X,Y,Z)^{-1}$.
{} We are interested in testing the conditional dependence relations
between $X$, $Y$, and $Z$%
. For a general overview of some of the many ways in which conditional
(in)dependence arises in statistics, we refer the reader to the classic
article by Dawid \cite{dawid1979conditional}. More recently, conditional
independence testing has found a central role in the areas of causal
inference and causal discovery; for a summary of this line of work,
see the recent text by Peters et al. \cite{peters2017elements}, as
well as the foundational works by Spirtes et al. \cite{spirtes2001causation}
and Pearl \cite{pearl2009causality}.

Numerous practical methods for testing conditional independence have
been proposed --- for instance, Heinze-Deml et al. \cite{heinze2018invariant}
describe six classes of conditional independence test which have been
implemented in the $\mathtt{R}$ package $\mathtt{CondIndTests}$.
Nonetheless, general theoretical guarantees for conditional independence
testing --- consistency results for general joint distributions $\mathcal{L}_{(X,Y,Z)}$,
sample complexity, power estimates, and so on --- have been hard
to come by. Recently, Shah and Peters \cite{shah2020hardness} have
demonstrated that this is not an accident: indeed, they show that
for general joint distribution $\mathcal{L}_{(X,Y,Z)}$, any valid
test of conditional independence does not have power against any alternative.
(This result has also been reproved, using techniques from optimal
transport, in \cite{neykov2020minimax}.) Consequently, it is of interest
to find \emph{mild} assumptions on $\mathcal{L}_{(X,Y,Z)}$ (or, similarly,
on the various conditional distributions between the variables $X$,
$Y$, and $Z$) under which a general test for conditional independence
does, indeed, exist. 

Here, our strategy is to present regularity conditions on the conditional
distributions $\mathcal{L}_{X|Z=z}$, $\mathcal{L}_{Y|Z=z}$, and
$\mathcal{L}_{(X,Y)|Z=z}$, so that if we bin the conditioning variable
$Z$ using a sufficiently fine, finite partition $V_{1},\ldots,V_{J}$
of $supp(\mathcal{L}_{Z})$, then for all $j=1,\ldots J$ and all
$z\in V_{j}$, 
\[
\mathcal{L}_{X|Z=z}\otimes\mathcal{L}_{Y|Z=z}\approx\mathcal{L}_{X|Z\in V_{j}}\otimes\mathcal{L}_{Y|Z\in V_{j}},\text{ and }\mathcal{L}_{(X,Y)|Z=z}\approx\mathcal{L}_{(X,Y)|Z\in V_{j}}
\]
where $\approx$ denotes closeness with respect to some suitable metric
on the space of probability measures. If this holds, then we can replace
the (intractable in general) problem of directly comparing the measures
$\mathcal{L}_{(X,Y)|Z=z}$ and $\mathcal{L}_{X|Z=z}\otimes\mathcal{L}_{Y|Z=z}$
for every $z$ in the range of $Z$, with the (costly, but in principle
tractable) problem of comparing $\mathcal{L}_{(X,Y)|Z\in V_{j}}$
and $\mathcal{L}_{X|Z\in V_{j}}\otimes\mathcal{L}_{Y|Z\in V_{j}}$
for every bin $V_{j}$ --- in other words, we seek regularity conditions
that allow for conditional independence to be tested indirectly using
multiple two-sample independence tests. Of course, one must also ensure
that the two-sample tests comparing $\mathcal{L}_{(X,Y)|Z\in V_{j}}$
and $\mathcal{L}_{X|Z\in V_{j}}\otimes\mathcal{L}_{Y|Z\in V_{j}}$
are themselves feasible (in the sense, for instance, of having explicit
finite-sample control on the Type I \& II error). 

Our work uses the formalism of \emph{optimal transport}, in particular
the \emph{Wasserstein distance} $W_{p}$, in an essential way. As
a metric structure on the space of probability measures, the Wasserstein
distance is only slightly stronger than weak convergence, so the stipulation
that (for instance) ``for small $V_{j}$, $W_{p}(\mathcal{L}_{(X,Y)|Z=z},\mathcal{L}_{(X,Y)|Z\in V_{j}})\approx0$''
is comparatively easy to satisfy. At the same time, it is also possible
to use the Wasserstein distance to conduct two-sample tests of quite
general probability measures. 

The structure of the paper is as follows. In Section \ref{sec:Two-sample-independence-testing},
we give an overview of the necessary background from optimal transport,
and discuss how to use the Wasserstein distance to conduct two-sample
independence tests. In Section \ref{sec:Conditional-independence-testing}
we review the general infeasibility of conditional independence testing,
but show that if the conditional distributions $\mathcal{L}_{X|Z=z}$,
$\mathcal{L}_{Y|Z=z}$, and $\mathcal{L}_{(X,Y)|Z=z}$ are \emph{Lipschitz}
functions of $z$ when the space of probability measures is equipped
with the $W_{p}$ metric, then binning the conditioning variable $Z$
incurs a small, explicit error in the $W_{p}$ metric. We then explain
how this allows for a ``$W_{p}$ multiple two-sample'' conditional
independence test, with explicit finite-sample control of the Type
I and Type II error, under a wide variety of mild auxiliary regularity
assumptions (which are needed so that the two-sample tests are themselves
feasible). Additionally, in Section \ref{sec:Plug-in-estimation}
we give a ``plug-in'' estimator for the $W_{p}$-Lipschitz constants
of $\mathcal{L}_{X|Z=z}$, $\mathcal{L}_{Y|Z=z}$, and $\mathcal{L}_{(X,Y)|Z=z}$. 

Regarding recent related work by others, while the work presented
in this article was in progress, the author learned of the recent
article \cite{neykov2020minimax}, which also considers conditional
independence testing under smoothness assumptions on the conditional
distributions $\mathcal{L}_{(X,Y)|Z}$, $\mathcal{L}_{X|Z}$, and
$\mathcal{L}_{Y|Z}$. The two works are similar in spirit, but complementary.
On the one hand, as discussed below in Section \ref{sec:Conditional-independence-testing},
our background assumptions are weaker than those used in \cite{neykov2020minimax};
on the other hand, \cite{neykov2020minimax} provides delicate minimax
estimates which are valid in the total variation setting.

\section{Two-sample independence testing with the Wasserstein distance\label{sec:Two-sample-independence-testing}}

We first consider testing whether $X\perp\!\!\!\perp Y$. Note that
this amounts to a two-sample test between the distributions $\mathcal{L}_{(X,Y)}$
and $\mathcal{L}_{X}\otimes\mathcal{L}_{Y}$. Furthermore, note that
given the ability to sample from $\mathcal{L}_{(X,Y)}$, we can always
artificially sample from $\mathcal{L}_{X}\otimes\mathcal{L}_{Y}$
as well --- for instance, one can sample from $\mathcal{L}_{(X,Y)}$
$2n$ times, and for the first $n$ samples discard the $Y$ values,
and for the second $n$ samples discard the $X$ values, and then
$(x_{1},y_{n+1}),(x_{2},y_{n+2}),\ldots,(x_{n},y_{2n})$ will be drawn
from $\mathcal{L}_{X}\otimes\mathcal{L}_{Y}$. 

Specifically, we consider Wasserstein-based two sample tests (cf.
\cite{ramdas2017wasserstein}, which considers some stronger results
in the univariate case). We do not attempt to give a general survey
on optimal transport/Wasserstein distances here (a more than sufficient
background reference is \cite{santambrogio2015optimal}; see also
\cite{panaretos2019statistical} for a survey tailor to a statistical
audience), but we give a short motivational discussion. Recall that
for $p\in[1,\infty)$, the $p$-Wasserstein distance on the space
$\mathcal{P}_{p}(U)$ of Borel probability measures with finite $p$-th
moments on the domain $U\subseteq\mathbb{R}^{d}$ is given by 
\[
W_{p}(\mu,\nu):=\inf_{\gamma\in\Pi(\mu,\nu)}\left(\int_{U^{2}}\Vert x-y\Vert^{p}d\gamma(x,y)\right)^{1/p}
\]
where $\Pi(\mu,\nu)$ is the space of all \emph{couplings} of $\mu$
and $\nu$, namely all probability measures on $U^{2}$ with marginals
$\mu$ and $\nu$. %
For $p\in[1,\infty)$, the $p$-Wasserstein distance metrizes the
weak convergence of probability measures together with convergence
of $p$th moments%
. In the jargon of optimal transport, a measure $\gamma\in\Pi(\mu,\nu)$
is said to be a \emph{transport plan }between $\mu$ and $\nu$, and
moreover $\gamma$ is said to be an \emph{optimal plan }if the infimum
in the definition of $W_{p}$ is attained at $\gamma$, that is, $W_{p}(\mu,\nu)=\left(\int_{U^{2}}\Vert x-y\Vert^{p}d\gamma(x,y)\right)^{1/p}$.
It is a basic feature of the theory that, given arbitrary measures
$\mu,\nu\in\mathcal{P}_{p}(U)$ on a ``reasonable'' domain $U$
(specifically, we require that $U$ be a Polish space), there exists
an optimal plan $\gamma$ between $\mu$ and $\nu$, but $\gamma$
need not be unique. 

If, moreover, we restrict our attention to \emph{absolutely continuous
measures} (say with respect to the Lebesgue measure on $U$), then
the $p$-Wasserstein distances enjoy another interpretation, in terms
of optimal transport maps rather than optimal couplings. If we write
$T_{\#}\mu$ for the pushforward of $\mu$ by the map $T$ (that is,
the measure defined by $T_{\#}\mu(A):=\mu(T^{-1}(A))$), then it turns
out that when $\mu$ is absolutely continuous,%
\[
W_{p}(\mu,\nu)=\inf_{T:T_{\#}\mu=\nu}\left(\int_{U}\Vert x-T(x)\Vert^{p}d\mu(x)\right)^{1/p}.
\]
There is an extensive theory studying these \emph{transport maps}
--- under what circumstances the infimum in the previous expression
is attained, whether the map $T$ instantiating the infimum is unique,
what regularity properties it possesses, and so on. This theoretical
apparatus will not play a role in our arguments. However, there is
one further feature of the $p$-Wasserstein distances that will be
be very helpful for us, which is the following. Suppose that $x_{1},x_{2},\ldots$
are i.i.d. samples from $\mu$. Let $\mu_{n}:=\frac{1}{n}\sum_{i=1}^{n}\delta_{x_{i}}$
denote the $n$th \emph{empirical measure }for $\mu$. (Note that
$\mu_{n}$ is random, but $\mu$ is fixed.) One variant of the law
of large numbers (namely, the Glivenko-Cantelli theorem) states that
$\mu_{n}$ converges to $\mu$ in distribution with probability 1;
accordingly, $W_{p}(\mu_{n},\mu)\rightarrow0$ almost surely, when
$p\in[1,\infty)$ (and, it turns out, when $p=\infty$ as well). In
fact, $W_{p}(\mu_{n},\mu)$ possesses known sample complexity bounds:
for instance, if the domain $U$ on which $\mu$ resides is compact,
$d$-dimensional and Hausdorff regular, and (for instance) $\mu$
is absolutely continuous w.r.t. the volume measure on $U$, then for
any $p\in[1,d/2)$ the $W_{p}$ sample complexity of $\mu$ is no
worse than $O(n^{-1/d})$ \cite[Theorem 1 and Proposition 8]{weed2019sharp}%
{} (see also discussion in the introduction of \cite{ambrosio2019finer}
or \cite{niles2019estimation} for a brief overview of similar results).
Significant improvements on this bound are possible in special cases,
for example if $\mu$ is actually supported on a lower-dimensional
surface \cite{weed2019sharp}. Note also that if $\mu\neq\nu$, then
the convergence of $W_{p}(\mu_{n},\nu_{n})$ to $W_{p}(\mu,\nu)$
was recently shown to occur at a faster rate than $O(n^{-1/d})$ in
general \cite{chizat2020faster,manole2021sharp}. %

Now let's return our focus to (two-sample) independence testing for
$X$ and $Y$. 

Let's be slightly more explicit about what this test looks like. The
null hypothesis $H_{0}$ will be that $X$ and $Y$ are indeed independent,
so $\mathcal{L}_{(X,Y)}=\mathcal{L}_{X}\otimes\mathcal{L}_{Y}$. We
first approximate both $\mathcal{L}_{(X,Y)}$ and $\mathcal{L}_{X}\otimes\mathcal{L}_{Y}$
with empirical measures based on samples, for instance 
\[
\frac{1}{n}\sum_{i=1}^{n}\delta_{(x_{i},y_{i})}\text{ and }\frac{1}{n}\sum_{i=1}^{n}\delta_{(x_{n+i},y_{2n+i})}
\]
respectively. Under the null hypothesis, we use the fact that $\mathcal{L}_{(X,Y)}=\mathcal{L}_{X}\otimes\mathcal{L}_{Y}$
to deduce that
\begin{multline*}
W_{p}\left(\frac{1}{n}\sum_{i=1}^{n}\delta_{(x_{i},y_{i})},\frac{1}{n}\sum_{i=1}^{n}\delta_{(x_{n+i},y_{2n+i})}\right)\\
\leq W_{p}\left(\frac{1}{n}\sum_{i=1}^{n}\delta_{(x_{i},y_{i})},\mathcal{L}_{(X,Y)}\right)+W_{p}\left(\mathcal{L}_{X}\otimes\mathcal{L}_{Y},\frac{1}{n}\sum_{i=1}^{n}\delta_{(x_{n+i},y_{2n+i})}\right)
\end{multline*}
and observe that%
\[
W_{p}\left(\frac{1}{n}\sum_{i=1}^{n}\delta_{(x_{i},y_{i})},\mathcal{L}_{(X,Y)}\right)\stackrel{n\rightarrow\infty}{\longrightarrow}0,\qquad W_{p}\left(\mathcal{L}_{X}\otimes\mathcal{L}_{Y},\frac{1}{n}\sum_{i=1}^{n}\delta_{(x_{n+i},y_{2n+i})}\right)\stackrel{n\rightarrow\infty}{\longrightarrow}0.
\]
This motivates the following definition:
\begin{defn}
\label{def:two-sample}($p$-Wasserstein Two Sample Independence Test)
The null hypothesis $H_{0}$ is that $\mathcal{L}_{(X,Y)}=\mathcal{L}_{X}\otimes\mathcal{L}_{Y}$.
We then set some ``level'' $\varepsilon_{0}>0$, and say that we
reject the null hypothesis if 
\[
W_{p}\left(\frac{1}{n}\sum_{i=1}^{n}\delta_{(x_{i},y_{i})},\frac{1}{n}\sum_{i=1}^{n}\delta_{(x_{n+i},y_{2n+i})}\right)\geq\varepsilon_{0},
\]
and we fail to reject otherwise. 
\end{defn}

In other words, the test above prescribes using the \emph{plug-in
estimator} $W_{p}(\widehat{\mathcal{L}_{(X,Y)}},\widehat{\mathcal{L}_{X}\otimes\mathcal{L}_{Y}})$
for $W_{p}(\mathcal{L}_{(X,Y)},\mathcal{L}_{X}\otimes\mathcal{L}_{Y})$. 
\begin{rem*}
Other finite-sample estimators for the $W_{p}$ distance between two
measures have been considered in the literature, and may even be more
desirable in some situations (see discussion, for instance, in \cite{chizat2020faster}).
We restrict our attention to the plug-in estimator only in the interest
of simplicity. Note also that there are theoretical obstructions to
producing a finite-sample estimator which is vastly superior to the
plug-in estimator, in particular the recent minimax result from \cite{singh2018minimax}.
\end{rem*}
To emphasize, the false rejection rate under the null is governed
by the previously stated $W_{p}$ sample complexity bounds for empirical
measures. On the other hand, the false acceptance rate under the alternative
(that is, the probability of the two empirical measures being within
$\varepsilon_{0}$ of each other, even though $\mathcal{L}_{(X,Y)}\neq\mathcal{L}_{X}\otimes\mathcal{L}_{Y}$)
depends on which distribution $\mathcal{L}_{(X,Y)}$ actually is;
more specifically, it depends on $W_{p}(\mathcal{L}_{(X,Y)},\mathcal{L}_{X}\otimes\mathcal{L}_{Y})$.

Before proceeding to testing procedures for conditional independence
in Section \ref{sec:Conditional-independence-testing} below, we give
a more explicit analysis of the ``Wasserstein two-sample'' independence
test described above. Namely: how must $\varepsilon_{0}$ be chosen
in order to ensure a given $p$-value and power function for the test?
Morally, this is ``just'' an application of the sample complexity
estimates for $W_{p}$ already mentioned; note that $\mathcal{L}_{(X,Y)}$
and $\mathcal{L}_{X}\otimes\mathcal{L}_{Y}$ reside in $\mathbb{R}^{d_{X}+d_{Y}}$,
so the ``off-the-shelf'' $W_{p}$ sample complexity for such a two-sample
test is $O(n^{-1/(d_{X}+d_{Y})})$ (for $p\in[1,\frac{1}{2}(d_{X}+d_{Y}))$).
{} However, this sample complexity is an \emph{asymptotic }result, and
is therefore unsuitable e.g. for the construction of an explicit confidence
interval. At the same time, while there is notable recent progress
on \emph{central limit theorems} for the empirical Wasserstein distance,
for instance for the asymptotic distribution of the Wasserstein distance
$W_{p}(\mu_{n},\nu)$ where $\mu_{n}$ is an empirical measure for
$\mu$ \cite{del2019central,del2021central}, the result therein are
also insufficient to construct explicit confidence intervals without
more information regarding the quantity $\mathbb{E}W_{p}(\mu_{n},\nu)$. 

We therefore proceed by an alternative analysis, which has essentially
two ingredients:
\begin{enumerate}
\item Given a measure $\mu$ and an empirical measure $\mu_{n}$ for $\mu$,
we desire an explicit upper bound on the \emph{expected error} introduced
by passing to the empirical measure $\mu_{n}$, that is, the quantity
$\mathbb{E}W_{p}(\mu,\mu_{n})$. In what follows, we call this the
\emph{expectation bound}.
\item Additionally, we would like a \emph{concentration inequality} indicating
how the random quantity $W_{p}(\mu,\mu_{n})$ is distributed around
its expectation $\mathbb{E}W_{p}(\mu,\mu_{n})$. 
\end{enumerate}
If we have these two ingredients, it is straightforward to produce
an upper bound on the probability that, under the null, the test indicated
in Definition \ref{def:two-sample} will nonetheless be rejected.

For an example of a suitable concentration inequality, let us mention
the following:
\begin{prop}[{\cite[Proposition 20]{weed2019sharp}}]
\label{prop:concentration} Let $(X,d)$ be a Polish metric space,
and let $\mu\in\mathcal{P}_{p}(X)$ be a measure supported on a set
of diameter at most $D$. Let $\mu_{n}$ denote an empirical measure
for $\mu$. Then, 
\[
\mathbb{P}\left[W_{p}^{p}(\mu,\mu_{n})\geq\mathbb{E}W_{p}^{p}(\mu,\mu_{n})+t\right]\leq\exp\left(-\frac{2nt^{2}}{D^{2p}}\right).
\]
\end{prop}

\begin{rem*}
Weed and Bach only state their concentration inequality in the case
where $D=1$. However the modification of their argument for general
diameter is quite routine; for completeness, we rerun the argument
with general diameter in Appendix \ref{sec:extras}.
\end{rem*}
It is also possible to extend the argument for Proposition \ref{prop:concentration}
to the case of a measure $\mu$ with unbounded support on a Polish
metric space $(X,d)$, at least provided $\mu$ also lies in $\mathcal{P}_{q}(X)$
for some $q>p$. This concentration inequality, which we state as
Proposition \ref{prop:unbounded Wp concentration ineq} below, is
of independent interest; we state and prove this result in Appendix
\ref{sec:extras}.

It turns out that of the two ingredients --- the expectation bound,
and the concentration inequality --- it is the former that is the
more challenging one, and indeed has been the subject of ongoing investigation
by numerous authors, such as \cite{boissard2014mean,dereich2013constructive,fournier2015rate,lei2020convergence,weed2019sharp}.
In particular, it is known that estimating the error introduced by
passing to empirical measures for distributions $\mathcal{L}_{(X,Y)}$
and $\mathcal{L}_{X}\otimes\mathcal{L}_{Y}$ depends sensitively on
the dimension $d_{X}+d_{Y}$. We note that the following theorem applies
directly for the case where $d_{X}+d_{Y}\geq3$; however, since there
is no absolute continuity requirement, even in the (important) case
where $d_{X}+d_{Y}=2$, this theorem can still be employed, simply
by taking $d=\min\{3,d_{X}+d_{Y}\}$. 
\begin{thm}[{\cite[Theorem 1]{dereich2013constructive}}]
\label{thm:expectation bound} Let $\mu$ be a measure on $\mathbb{R}^{d}$
with $d\geq3$, and let $\mu_{n}$ be an empirical measure for $\mu$.
Let $p\in[1,d/2)$ and $q>dp/(d-p)$. Then there exists a constant
$\kappa_{p,q,d}$ depending only on $p,q,$ and $d$, such that 
\[
\mathbb{E}[W_{p}^{p}(\mu_{n},\mu)]^{1/p}\leq\kappa_{p,q,d}\left[\int_{\mathbb{R}^{d}}\Vert x\Vert^{q}d\mu(x)\right]^{1/q}n^{-1/d}
\]
where $\Vert\cdot\Vert$ is a norm on $\mathbb{R}^{d}$ and $\kappa_{p,q,d}$
is given explicitly in \cite{dereich2013constructive}.
\end{thm}

\begin{rem*}
While our focus in this article is to construct a test with finite-sample
guarantees under very general assumptions, we also alert the reader
to the fact that the rate in the theorem above can be improved significantly
if $\mu$ is known to have some additional regularity, e.g. is concentrated
near a low-dimensional set or is a Gaussian mixture --- see extensive
consideration of these issues in \cite{weed2019sharp}.
\end{rem*}
Observe that by combining the estimates 
\[
\mathbb{P}[|W_{p}^{p}(\mu_{n},\mu)-\mathbb{E}[W_{p}^{p}(\mu_{n},\mu)]|\geq t]\leq\exp(-2nt^{2}/D^{2p})
\]
and 
\[
\mathbb{E}[W_{p}^{p}(\mu_{n},\mu)]^{1/p}\leq\kappa_{p,q,d}\left[\int_{\mathbb{R}^{d}}\Vert x\Vert^{q}d\mu(x)\right]^{1/q}n^{-1/d}
\]
it follows directly that in the case where $\mu$ has support with
diameter at most $D$ (and therefore $\int_{\mathbb{R}^{d}}\Vert x\Vert^{q}d\mu(x)<\infty$),
\begin{align*}
\mathbb{P}\left[W_{p}^{p}(\mu_{n},\mu)\geq t+\mathbb{E}[W_{p}^{p}(\mu_{n},\mu)]\right] & \leq\mathbb{P}\left[W_{p}^{p}(\mu_{n},\mu)\geq t+\kappa_{p,q,d}^{p}\left[\int_{\mathbb{R}^{d}}\Vert x\Vert^{q}d\mu(x)\right]^{p/q}n^{-p/d}\right]\\
 & \leq\exp(-2nt^{2}/D^{2p}).
\end{align*}
For instance (although this is far from the most efficient upper bound
in $n$!), if $n$ is large enough that $\kappa_{p,q,d}^{p}\left[\int_{\mathbb{R}^{d}}\Vert x\Vert^{q}d\mu(x)\right]^{p/q}n^{-p/d}\leq\frac{1}{2}\left(\frac{\varepsilon}{2}\right)^{p}$,
then putting $t=\frac{1}{2}\left(\frac{\varepsilon}{2}\right)^{p}$
gives 
\[
\mathbb{P}\left[W_{p}^{p}(\mu_{n},\mu)\geq\left(\frac{\varepsilon}{2}\right)^{p}\right]\leq\exp\left(-\frac{2n\varepsilon^{2p}}{4^{p+1}D^{2p}}\right)
\]
or equivalently 
\[
\mathbb{P}\left[W_{p}(\mu_{n},\mu)\geq\frac{\varepsilon}{2}\right]\leq\exp\left(-\frac{2n\varepsilon^{2p}}{4^{p+1}D^{2p}}\right).
\]
Lastly, note that if $\mu_{n}$ and $\mu_{n}^{\prime}$ are both empirical
measures drawn from the same distribution $\mu$, and $W_{p}(\mu_{n},\mu_{n}^{\prime})\geq\varepsilon$,
then it must be the case that either $W_{p}(\mu_{n},\mu)\geq\frac{\varepsilon}{2}$
or $W_{p}(\mu_{n}^{\prime},\mu)\geq\frac{\varepsilon}{2}$. Hence,
\[
\mathbb{P}\left[W_{p}(\mu_{n},\mu_{n}^{\prime})\geq\varepsilon\right]\leq2\exp\left(-\frac{2n\varepsilon^{2p}}{4^{p+1}D^{2p}}\right).
\]

Likewise, if $\mu$ has unbounded support but still satisfies a higher
$q$th moment bound, as per Theorem \ref{thm:expectation bound},
one can combine Theorems \ref{prop:unbounded Wp concentration ineq}
and \ref{thm:expectation bound} in a similar manner: see Corollary
\ref{cor:d>=00003D3 concentration + exp. bound} below.

It is also possible to explicitly bound the type II error of the test
given in Definition \ref{def:two-sample}, as a function of $W_{p}(\mathcal{L}_{(X,Y)},\mathcal{L}_{X}\otimes\mathcal{L}_{Y})$.
Indeed, suppose that $W_{p}(\mathcal{L}_{(X,Y)},\mathcal{L}_{X}\otimes\mathcal{L}_{Y})\geq\varepsilon_{0}+\delta_{0}$.
Then, if a type II error occurs, namely $W_{p}\left(\frac{1}{n}\sum_{i=1}^{n}\delta_{(x_{i},y_{i})},\frac{1}{n}\sum_{i=1}^{n}\delta_{(x_{n+i},y_{2n+i})}\right)<\varepsilon_{0}$,
it must be the case that at least one of the following holds: 
\[
W_{p}\left(\mathcal{L}_{(X,Y)},\frac{1}{n}\sum_{i=1}^{n}\delta_{(x_{i},y_{i})}\right)>\frac{\delta_{0}}{2}\text{ or }W_{p}\left(\mathcal{L}_{X}\otimes\mathcal{L}_{Y},\frac{1}{n}\sum_{i=1}^{n}\delta_{(x_{n+i},y_{2n+i})}\right)>\frac{\delta_{0}}{2}.
\]
But we can upper bound the probability of each of these events using
exactly the same analysis --- expectation bound combined with concentration
inequality --- that we have just described. 

Additionally, let us remark on the task of actually computing the
estimator $W_{p}(\hat{\mu}_{n},\hat{\nu}_{n})$ for $W_{p}(\mu,\nu)$.
Computing the optimal coupling for $\hat{\mu}_{n}$ and $\hat{\nu}_{n}$
takes the form of a classical (discrete) \emph{optimal matching} problem,
which has been extensively studied and is amenable to the standard
methods of linear programming. State-of-the-art computational methods
for this problem are surveyed, for instance, in \cite{peyre2019computational};
for example, fast methods based on \emph{Sinkhorn's algorithm}, discussed
therein, achieve a computational cost of $O(n^{2}\log n)$. Therefore,
when conducting a Wasserstein two-sample test, one must take care
to work with empirical measures supported on a number $n$ of points
which is large enough that the test achieves a small enough $p$-value,
but not so large that the estimator itself is infeasible to compute.

In what follows, we (to some extent) work \emph{modulo} the choice
of a Wasserstein two-sample (independence) test with a given $p$-value
and power function. However, we hope that from the previous discussion,
the reader is convinced that such a test is feasible, albeit subject
to ongoing technical improvements, \emph{viz}. in terms of expectation
bounds, concentration inequalities, and computational complexity. 

\section{Conditional independence testing via binning\label{sec:Conditional-independence-testing}}

Now we consider conditional independence testing. That is, we ask
whether $X\perp\!\!\!\perp Y\mid Z$. At the level of distributions,
this amounts to comparing $\mathcal{L}_{(X,Y)|Z=z}$ with $\mathcal{L}_{X|Z=z}\otimes\mathcal{L}_{Y|Z=z}$
for \emph{every} $z$ in the range of $Z$. (Note that these distributions
should be understood as disintegrations of the joint distributions
$\mathcal{L}_{(X,Y,Z)}$ and $\mathcal{L}_{(X,Z)}$ and $\mathcal{L}_{(Y,Z)}$
w.r.t $Z$, e.g. in the sense of \emph{regular conditional distributions}.)
That is, conditional independence testing may be understood as (at
least implicitly) requiring a \emph{continuum} number of two-sample
tests. 

It should now be clear to the reader that without further assumptions
on the random variables $X,Y$, and $Z$ (and their joint distribution),
finite sample conditional independence testing is futile. Indeed,
it may be that for all of the (finitely many) values $z$ of $Z$
that we actually get to sample, it is the case that $\mathcal{L}_{(X,Y)|Z=z}=\mathcal{L}_{X|Z=z}\otimes\mathcal{L}_{Y|Z=z}$;
but for some $z^{\prime}$ out of our sample, $\mathcal{L}_{(X,Y)|Z=z^{\prime}}$
and $\mathcal{L}_{X|Z=z^{\prime}}\otimes\mathcal{L}_{Y|Z=z^{\prime}}$
are very far apart as distributions. Compare with the analogous concern
for $\mathcal{L}_{(X,Y)}$ versus $\mathcal{L}_{X}\otimes\mathcal{L}_{Y}$;
here, we have a guarantee from the (convergence of measures version
of the) law of large numbers that empirical measures converge to the
measures they are drawn from almost surely, with a deviation rate/concentration
bounds given by Sanov's theorem or similar. We have no such luck here,
simply from the fact that $\mathcal{L}_{(X,Y)|Z=z}$ and $\mathcal{L}_{(X,Y)|Z=z^{\prime}}$
are different distributions with no \emph{a priori} relationship.
{} (The proof in \cite{shah2020hardness} which shows that conditional
independence testing is not feasible in general, proceeds by way of
a different, but in some sense related, argument: there, they construct
two joint distributions of random variables $(X,Y,Z)$ and $(\tilde{X},\tilde{Y},\tilde{Z})$
where each of $\{X,\tilde{X}\}$, $\{Y,\tilde{Y}\}$, and $\{Z,\tilde{Z}\}$
are close with high probability, hence costly to distinguish with
finite samples, yet $(X,Y,Z)$ and $(\tilde{X},\tilde{Y},\tilde{Z})$
have different conditional dependence relations.) 

So in order to construct a conditional independence test which, say,
comes with any kind of finite sample guarantee, some auxiliary assumptions
are necessary. In particular, if we knew that all the marginal distributions
varied in some continuous fashion with $z$, we would be able to use
the fact that $\mathcal{L}_{(X,Y)|Z=z}\approx\mathcal{L}_{(X,Y)|Z=z^{\prime}}$
whenever $z\approx z^{\prime}$ (and likewise for the families of
distributions $\mathcal{L}_{X|Z=z}$ and $\mathcal{L}_{Y|Z=z}$),
to deduce some overall estimate of conditional dependence by binning
the variable $Z$, and then performing a (possibly very large) number
of separate two-sample independence tests. 

In what follows, we consider the special case where $\mathcal{L}_{Z}$
is compactly supported inside $\mathbb{R}^{d_{Z}}$;%
{} however, we do not assume compact support for the conditional distributions
$\mathcal{L}_{(X,Y)|Z=z}$. 

We also make the following type of assumption: consider the function
\[
\mathbb{R}^{d_{Z}}\rightarrow\mathcal{P}(\mathbb{R}^{d_{X}+d_{Y}})
\]
\[
z\mapsto\mathcal{L}_{(X,Y)|Z=z}.
\]
We ask that $z\mapsto\mathcal{L}_{(X,Y)|Z=z}$ is Lipschitz continuous%
, provided we equip $\mathcal{P}(\mathbb{R}^{d_{X}+d_{Y}})$ with
the $W_{p}$ metric; in other words, 
\begin{multline*}
z\mapsto\mathcal{L}_{(X,Y)|Z=z}\text{ is \ensuremath{L_{XY}}-Lipschitz}\iff\\
(\forall z,z^{\prime}\in\mathbb{R}^{d_{Z}}\cap supp(\mathcal{L}_{Z}))W_{p}(\mathcal{L}_{(X,Y)|Z=z},\mathcal{L}_{(X,Y)|Z=z^{\prime}})\leq L_{XY}|z-z^{\prime}|.
\end{multline*}

Observe that requiring that $z\mapsto\mathcal{L}_{(X,Y)|Z=z}$ be
$L$-Lipschitz not only imposes that $z\approx z^{\prime}\implies\mathcal{L}_{(X,Y)|Z=z}\approx_{W_{p}}\mathcal{L}_{(X,Y)|Z=z^{\prime}}$,
it does so in a quantitative fashion: if we know that $|z-z^{\prime}|<\varepsilon/L_{XY}$,
then we can say that $W_{p}(\mathcal{L}_{(X,Y)|Z=z},\mathcal{L}_{(X,Y)|Z=z})<\varepsilon$. 

In addition, we will also require that the functions 
\[
z\mapsto\mathcal{L}_{X|Z=z}\text{ and }z\mapsto\mathcal{L}_{Y|Z=z}
\]
be Lipschitz continuous, in other words, 
\[
(\forall z,z^{\prime}\in\mathbb{R}^{d_{Z}}\cap supp(\mathcal{L}_{Z}))\quad W_{p}(\mathcal{L}_{X|Z=z},\mathcal{L}_{X|Z=z^{\prime}})\leq L_{X}|z-z^{\prime}|;
\]
\[
(\forall z,z^{\prime}\in\mathbb{R}^{d_{Z}}\cap supp(\mathcal{L}_{Z}))\quad W_{p}(\mathcal{L}_{Y|Z=z},\mathcal{L}_{Y|Z=z^{\prime}})\leq L_{Y}|z-z^{\prime}|.
\]
In what follows, it turns out that we will need to set some small
parameters (specifically the diameter of bins for the range of $Z$)
to be small compared to the inverse of \emph{all three} optimal Lipschitz
constants simultaneously. %

\begin{rem*}
There are, of course, \emph{many} inequivalent metrics of interest
on the space of probability measures. Why choose the $W_{p}$ metric
(for some $p\in[1,\infty)$) in particular? Indeed, the recent article
\cite{neykov2020minimax}, also concerned with conditional independence
testing, makes extensive use of the total variation ($TV$) metric
rather than a metric coming from an optimal transport problem. To
some extent, the choice of metric on the space of probability measures
is simply a modeling decision, but let us emphasize a topological
point in favor of the $W_{p}$ metrics. On compact sets, convergence
in $W_{p}$ is equivalent to convergence in distribution, a.k.a. convergence
in the probabilist's weak topology; on the other hand, the TV norm
metrizes \emph{strong} convergence, equivalently convergence in $L^{1}$.
Since the strong topology is (as the name suggests) stronger than
the weak topology, it follows that \emph{fewer} functions from $\mathbb{R}^{d}$
into the space of probability measures are strongly continuous than
weakly continuous. In our particular situation, this means that on
a compact domain, the requirement that conditional distributions of
the form $z\mapsto\mathcal{L}_{(X,Y)|Z=z}$ be $W_{p}$-continuous
is less stringent than requiring that $z\mapsto\mathcal{L}_{(X,Y)|Z=z}$
be $TV$-continuous. This difference in strength of assumption is
especially evident in the case where the conditional distributions
$z\mapsto\mathcal{L}_{(X,Y)|Z=z}$ (and $z\mapsto\mathcal{L}_{X|Z=z}$
and $z\mapsto\mathcal{L}_{Y|Z=z}$) are \emph{mutually singular} for
different values of $z$, as is the case, for example, when the conditional
distributions are supported on disjoint hypersurfaces parametrized
by $z$. For emphasis, we give a couple of concrete instances of this
type in the following example.
\end{rem*}
\begin{example}
(examples of conditional distribution which is $W_{p}$-Lipschitz
but not $TV$-continuous) (1) Consider the slightly trivial case where
$Y=X$ and $Z=X$. In this case, $\mathcal{L}_{(X,Y)|Z=z}=\mathcal{L}_{X|Z=z}\otimes\mathcal{L}_{Y|Z=z}=\delta_{z}\otimes\delta_{z}$.
In particular, if $(z_{n})$ is a sequence which converges to $z$
in $\mathbb{R}^{d_{Z}}$ (but $z_{n}\neq z$ for all $n\in\mathbb{N}$),
then $(\delta_{z_{n}}\otimes\delta_{z_{n}})$ converges to $\delta_{z}\otimes\delta_{z}$
in the weak topology as well as in $W_{p}$ (for every $p\in[1,\infty)$).
However, $TV(\delta_{z_{n}}\otimes\delta_{z_{n}},\delta_{z}\otimes\delta_{z})\not\rightarrow0$,
since the total variation distance between \emph{any} two mutually
singular measures is identically $1$. 

(2) Suppose that $X$ and $Y$ are both random variables in $\mathbb{R}^{2}$
(equipped with polar coordinates) where the radial variable for $X$
and $Y$ is deterministically coupled. In other words, there is some
random variable $Z:\Omega\rightarrow\mathbb{R}_{+}$ where $X=(Z,\Theta_{1}(Z,\cdot))$
and $Y=(Z,\Theta_{2}(Z,\cdot)))$ where $\Theta_{1}(r,\cdot)$ and
$\Theta_{2}(r,\cdot)$ are both ``random angles'', that is, random
variables from $\Omega$ into $[0,2\pi)$, which may depend on the
radial variable $r$. In this case, $X$ and $Y$ may both have full
support in $\mathbb{R}^{2}$; but the conditional distributions $\mathcal{L}_{X|Z=z}$
and $\mathcal{L}_{Y|Z=z}$ are always supported within the sphere
of radius $z$, and so $\mathcal{L}_{X|Z=z}$ and $\mathcal{L}_{Y|Z=z^{\prime}}$
are mutually singular for all distinct $z$ and $z^{\prime}$ (and
similarly for $\mathcal{L}_{(X,Y)|Z=z}$). It follows that $z\mapsto\mathcal{L}_{(X,Y)|Z=z}$
and $z\mapsto\mathcal{L}_{X|Z=z}$ and $z\mapsto\mathcal{L}_{Y|Z=z}$
cannot be $TV$-continuous, but may still be $W_{p}$-continuous if
$\Theta_{1}(r,\cdot)$ and $\Theta_{2}(r,\cdot)$ vary continuously
with $r$ (we demonstrate this immediately below, in Example \ref{exa:additive noise model},
in the particular case where $\Theta_{1}(r,\cdot)$ and $\Theta_{2}(r,\cdot)$
are Lipschitz continuous in $r$).
\end{example}

It is certainly worth asking how $W_{p}$-Lipschitz type assumptions
relate to the candidate functional relations between $X$, $Y$, and
$Z$ in the causal inference setting. For instance, if there is no
deterministic functional relationship between $X$ and $Y$ but %
{} $X=f(Z)+E_{1}$ and $Y=g(Z)+E_{2}$ (where $E_{1}$ and $E_{2}$
are i.i.d. noise variables), how do the $W_{p}$-Lipschitz assumptions
on the conditional distributions we have just detailed relate to the
regularity properties of $f$ and $g$? 
\begin{example}[Additive noise model; $Z$ causes $X$ and $Y$ but $X\perp\!\!\!\perp Y\mid Z$]
\label{exa:additive noise model} Suppose we are given random variables
$X,Y$, and $Z$, and we know that $X=f(Z)+E_{1}$ and $Y=g(Z)+E_{2}$,
where $E_{1}$ and $E_{2}$ are independent and identically distributed
noise variables (for instance, $E_{1}$ and $E_{2}$ may be (truncated)
Gaussians), and $f$ and $g$ are (deterministic) Lipschitz continuous
functions. We claim that in this circumstance, all of $z\mapsto\mathcal{L}_{X|Z=z}$,
$z\mapsto\mathcal{L}_{Y|Z=z}$, and $z\mapsto\mathcal{L}_{(X,Y)|Z=z}$
are Lipschitz continuous with respect to $W_{p}$. Note that in this
case, if we condition on $Z=z_{0}$, then $X$ has the distribution
$f(z_{0})+E_{1}$ and $Y$ has the distribution $g(z_{0})+E_{2}$.
Thus, if both $f$ and $g$ are $L$-Lipschitz, the distributions
of $X|Z=z_{0}$ and $X|Z=z_{1}$ are identical, \emph{modulo} a shift
by a constant at most $\varepsilon$, provided that $|z_{0}-z_{1}|<\varepsilon/L$,
and similarly for $Y|Z=z_{0}$ and $Y|Z=z_{1}$ respectively. Note
that since $\mathcal{L}_{X|Z=z_{0}}=f(z_{0})+\mathcal{L}_{E_{1}}$,
it follows that the shift map
\[
x\mapsto x+f(z_{1})-f(z_{0})
\]
pushes $\mathcal{L}_{X|Z=z_{0}}$ onto $\mathcal{L}_{X|Z=z_{1}}$.
Consequently, this shift is an admissible transport map in the optimal
transport problem, and it follows that 
\begin{align*}
W_{p}(\mathcal{L}_{X|Z=z_{0}},\mathcal{L}_{X|Z=z_{1}}) & \leq\left(\int_{\Omega}|f(z_{1})-f(z_{0})|^{p}d\mathcal{L}_{X|Z=z_{0}}\right)^{1/p}.
\end{align*}
But since $|f(z_{1})-f(z_{0})|<\varepsilon$, it follows that $W_{p}(\mathcal{L}_{X|Z=z_{0}},\mathcal{L}_{X|Z=z_{1}})<\varepsilon$.%
{} By an identical argument, it also holds that $W_{p}(\mathcal{L}_{Y|Z=z_{0}},\mathcal{L}_{Y|Z=z_{1}})<\varepsilon$.
And, of course, in this case, $X$ and $Y$ are, in fact conditionally
independent given $Z$, so since $\mathcal{L}_{(X,Y)|z}=\mathcal{L}_{X|z}\otimes\mathcal{L}_{Y|z}$
we see that an an analogous estimate applies to the joint distribution;
that is, that since the product shift map 
\[
(x,y)\mapsto(x+f(z_{1})-f(z_{0}),y+g(z_{1})-g(z_{0}))
\]
pushes $\mathcal{L}_{X|z_{0}}\otimes\mathcal{L}_{Y|z_{0}}$ onto $\mathcal{L}_{X|z_{1}}\otimes\mathcal{L}_{Y|z_{1}}$,
we have the transport map upper bound
\begin{align*}
W_{p}^{p}(\mathcal{L}_{(X,Y)|z_{0}},\mathcal{L}_{(X,Y)|z_{1}}) & \leq\int_{\Omega^{2}}|(f(z_{1}),g(z_{1}))-(f(z_{0}),g(z_{0}))|^{p}d\mathcal{L}_{X|Z=z_{0}}\otimes d\mathcal{L}_{Y|Z_{0}}\\
 & \leq2^{p-1}\left(\int_{\Omega^{2}}|(f(z_{1})-f(z_{0})|^{p}+|g(z_{1})-g(z_{0}))|^{p}d\mathcal{L}_{X|Z=z_{0}}\otimes d\mathcal{L}_{Y|Z_{0}}\right)\\
 & =2^{p-1}\int_{\Omega}|(f(z_{1})-f(z_{0})|^{p}d\mathcal{L}_{X|Z=z_{0}}+2^{p-1}\int_{\Omega}|g(z_{1})-g(z_{0}))|^{p}d\mathcal{L}_{Y|Z_{0}}\\
 & \leq(2\varepsilon)^{p}.
\end{align*}
{} %
\end{example}

\begin{rem*}
On the other hand, if we consider a non-additive noise model, then
it is less clear whether there is clean a relationship between the
Lipschitz property of $z\mapsto\mathcal{L}_{(X,Y)|z}$ and the functional
form of the causal relations between $X,Y,$ and $Z$. 
\end{rem*}
As promised, the Lipschitz assumption allows us to perform a conditional
independence test on $(X,Y,Z)$ using binning. Let us now summarize
how this works, before proceeding with a series of technical arguments.
The idea is to divide the range of $Z$ into cells $V_{j}$ of diameter
at most $\varepsilon/L$ where $L$ is some function of $L_{X}$,
$L_{Y}$, and $L_{XY}$ (here one could use, say, cubes or Voronoi
cells), and compute empirical estimates of the conditional distributions
$\mathcal{L}_{(X,Y)|z\in V_{j}}$ and $\mathcal{L}_{X|z\in V{}_{j}}\otimes\mathcal{L}_{Y|z\in V_{j}}$.
Then, for each cell $V_{j}$, we compute the Wasserstein distance
between those two empirical distributions; if for each $V_{j}$ the
distance between the two empirical distributions is small, then we
can conclude that with high probability, $X$ and $Y$ are conditionally
independent given ``$Z$ discretized onto the cells $V_{j}$''.
A fancier way to say this is that if $\mathcal{V}$ is the $\sigma$-algebra
in the sample space generated by the partition $Z^{-1}\{V_{j}\}$,
we are testing whether $X\perp\!\!\!\perp Y\mid\mathbb{E}[Z\mid\mathcal{V}]$.
Finally, our Lipschitz assumption will allow us to deduce a relationship
between our test for $X\perp\!\!\!\perp Y\mid\mathbb{E}[Z\mid\mathcal{V}]$,
and whether $X\perp\!\!\!\perp Y\mid Z$. %

Crucially, we must explicitly control the error introduced by replacing
$Z$ with $\mathbb{E}[Z\mid\mathcal{V}]$. %
At the heuristic level, %
{} the situation is rather straightforward. Indeed, suppose that our
joint and conditional distributions have densities; consider $(x_{1},y_{1},z_{1})$
and $(x_{2},y_{2},z_{2})$, both samples from the joint density $p(x,y,z)$.
Suppose that $z_{1}$ and $z_{2}$ belong to the same cell, and thus
$z_{2}\in B_{\varepsilon/L}(z_{1})$. We want to pretend that up to
some small error, $(x_{2},y_{2})$ is a sample from $p(x,y|z_{1})$
rather than $p(x,y|z_{2})$, as is actually the case. If we had an
``optimal transport oracle'', we could ask for the optimal transport
map $T$ from $p(x,y|z_{2})$ to $p(x,y|z_{1})$, and then just compute
$T((x_{2},y_{2}))$; then this \emph{would} be a sample from $p(x,y|z_{1})$.%
{} Instead, we do nothing to $(x_{2},y_{2})$, but since $p(x,y|z_{2})$
and $p(x,y|z_{1})$ are close in the Wasserstein sense, we know that
$\Vert T-id\Vert$ is small%
; this will allow us to say something like
\[
W_{p}\left(\frac{1}{n}\sum_{i=1}^{n}\delta_{(x_{i},y_{i})|z_{i}\in V_{j}},p(x,y|z_{1})\right)\leq\varepsilon.
\]

 Despite this pleasing heuristic motivation, phrased in terms of
optimal maps and samples, it turns out to be more direct, from a technical
standpoint, to work with optimal plans and population measures, which
is how we will proceed in the remainder of the section. Our analysis
depends on the following lemma (stated with under rather general assumptions,
since the proof is no more difficult):
\begin{lem}
\label{lem:disintegration}Let $z\mapsto\mu_{z}$ be a Borel measurable
function from a %
{} Polish space $S$, to $(\mathcal{P}_{p}(U),W_{p})$, where $(U,d)$
is a Polish metric space, and let $V\subseteq S$ be a Borel set inside
$S$. Let $\lambda$ be a probability measure on $V$, and let $\nu\in W_{p}(U)$.
Then, 
\[
W_{p}^{p}\left(\int_{V}\mu_{z}d\lambda(z),\nu\right)\leq\int_{V}W_{p}^{p}(\mu_{z},\nu)d\lambda(z).
\]
\end{lem}

In fact, the proof of such a result is more or less indicated by the
contents of \cite[Ch. 4 \& 5]{villani2008optimal}; however, the argument
therein omits a number of technical points, so we nonetheless give
a rather explicit proof in Appendix \ref{Proof of disintegration lemma}. 

We can now easily prove the following:
\begin{prop}
\label{prop:binning-1}Suppose that $z\mapsto\mathcal{L}_{(X,Y)|Z=z}$
is $L_{XY}$-Lipschitz with respect to $W_{p}$. Suppose that $\{V_{j}\}_{j=1}^{J}$
is a measurable partition of the support of $\mathcal{L}_{Z}$, and
that $diam(V_{j})\leq\varepsilon/L_{XY}$ for every bin $V_{j}$ in
the partition. Then, for all $z_{0}\in V_{j}$, it holds that
\[
W_{p}(\mathcal{L}_{(X,Y)|Z\in V_{j}},\mathcal{L}_{(X,Y)|Z=z_{0}})\leq\varepsilon.
\]
\end{prop}

\begin{proof}
In the setting of Lemma \ref{lem:disintegration}, let $V=V_{j}$,
$U=\mathbb{R}^{d_{X}}\times\mathbb{R}^{d_{Y}}$, and
\[
\mu_{z}=\mathcal{L}_{(X,Y)|Z=z}\text{ and }\nu=\mathcal{L}_{(X,Y)|Z=z_{0}}.
\]
Furthermore, let $\lambda=\frac{1}{\mathcal{L}_{Z}(V_{j})}\mathcal{L}_{Z}\upharpoonright V_{j}$.
Under the assumptions of the proposition, 
\[
W_{p}^{p}\left(\mathcal{L}_{(X,Y)|Z=z},\mathcal{L}_{(X,Y)|Z=z_{0}}\right)\leq\varepsilon^{p}\text{ for all }z,z_{0}\in V_{j}.
\]
Since the function $z\mapsto\mu_{z}$ is Lipschitz continuous from
$supp(\mathcal{L}_{Z})$ (which is closed inside $\mathbb{R}^{d_{Z}}$,
hence a Polish space) to $(\mathcal{P}_{p}(\mathbb{R}^{d_{X}+d_{Y}}),W_{p})$,
it is automatically Borel measurable. Lemma \ref{lem:disintegration}
then indicates that 
\begin{align*}
W_{p}^{p}\left(\fint_{V_{j}}\mathcal{L}_{(X,Y)|Z=z}d\mathcal{L}_{Z}(z),\mathcal{L}_{(X,Y)|Z=z_{0}}\right) & \leq\fint_{V_{j}}W_{p}^{p}\left(\mathcal{L}_{(X,Y)|Z=z},\mathcal{L}_{(X,Y)|Z=z_{0}}\right)d\mathcal{L}_{Z}(z)\\
 & \leq\fint_{V_{j}}\varepsilon^{p}d\mathcal{L}_{Z}(z)\\
 & =\varepsilon^{p}.
\end{align*}
Since $\fint_{V_{j}}\mathcal{L}_{(X,Y)|Z=z}d\mathcal{L}_{Z}(z)=\mathcal{L}_{(X,Y)|Z\in V_{j}}$,
the proposition is proved.
\end{proof}
We also need an analogous proposition, but for the measures $\mathcal{L}_{X|Z=z_{0}}\otimes\mathcal{L}_{Y|Z=z_{0}}$
and $\mathcal{L}_{X|Z\in V_{j}}\otimes\mathcal{L}_{Y|Z\in V_{j}}$.
\begin{prop}
\label{prop:binning-2}Suppose that $z\mapsto\mathcal{L}_{X|Z=z}$
and $z\mapsto\mathcal{L}_{Y|Z=z}$ are $L_{X}$- and $L_{Y}$-Lipschitz
(resp.) with respect to $W_{p}$. Suppose that $\{V_{j}\}_{j=1}^{J}$
is a measurable partition of the support of $\mathcal{L}_{Z}$, and
that for all bins $V_{j}$,
\[
diam(V_{j})\leq\frac{\varepsilon}{\left(2^{p-1}\left(L_{X}^{p}+L_{Y}^{p}\right)\right)^{1/p}}.
\]
Then %
\[
W_{p}\left(\mathcal{L}_{X|Z\in V_{j}}\otimes\mathcal{L}_{Y|Z\in V_{j}},\mathcal{L}_{X|Z=z_{0}}\otimes\mathcal{L}_{Y|Z=z_{0}}\right)\leq\varepsilon.
\]
\end{prop}

\begin{proof}
The argument is similar to that for the previous proposition, but
with a couple of additional complications. 

Let $\gamma_{X}$ be an optimal transport plan between $\mathcal{L}_{X|Z\in V_{j}}$
and $\mathcal{L}_{X|Z=z_{0}}$, and similarly for $\gamma_{Y}$. Then,%
{} $\gamma_{X}\otimes\gamma_{Y}$ is a transport plan between $\mathcal{L}_{X|Z\in V_{j}}\otimes\mathcal{L}_{Y|Z\in V_{j}}$
and $\mathcal{L}_{X|Z=z_{0}}\otimes\mathcal{L}_{Y|Z=z_{0}}$; consequently,
\begin{multline*}
W_{p}^{p}\left(\mathcal{L}_{X|Z\in V_{j}}\otimes\mathcal{L}_{Y|Z\in V_{j}},\mathcal{L}_{X|Z=z_{0}}\otimes\mathcal{L}_{Y|Z=z_{0}}\right)\\
\begin{aligned} & \leq\int_{\left(\mathbb{R}^{d_{X}+d_{Y}}\right)^{2}}\Vert(x,y)-(x_{0},y_{0})\Vert^{p}d(\gamma_{X}\otimes\gamma_{Y})((x,y),(x_{0},y_{0}))\\
 & =\int_{\left(\mathbb{R}^{d_{X}}\right)^{2}}\int_{\left(\mathbb{R}^{d_{Y}}\right)^{2}}\Vert(x,y)-(x_{0},y_{0})\Vert^{p}d\gamma_{X}(x,x_{0})d\gamma_{Y}(y,y_{0})\\
 & =\int_{\left(\mathbb{R}^{d_{X}}\right)^{2}}\int_{\left(\mathbb{R}^{d_{Y}}\right)^{2}}\left(\sqrt{(x-x_{0})^{2}+(y-y_{0})^{2}}\right)^{p}d\gamma_{X}(x,x_{0})d\gamma_{Y}(y,y_{0}).
\end{aligned}
\end{multline*}
Now, since $(x-x_{0})^{2}+(y-y_{0})^{2}\leq(x-x_{0})^{2}+2|x-x_{0}||y-y_{0}|+(y-y_{0})^{2}$,
it follows that 
\[
\sqrt{(x-x_{0})^{2}+(y-y_{0})^{2}}\leq\sqrt{\left(|x-x_{0}|+|y-y_{0}|\right)^{2}}
\]
and so
\[
\int_{\left(\mathbb{R}^{d_{X}}\right)^{2}}\int_{\left(\mathbb{R}^{d_{Y}}\right)^{2}}\left(\sqrt{(x-x_{0})^{2}+(y-y_{0})^{2}}\right)^{p}d\gamma_{X}d\gamma_{Y}\leq\int_{\left(\mathbb{R}^{d_{X}}\right)^{2}}\int_{\left(\mathbb{R}^{d_{Y}}\right)^{2}}\left(|x-x_{0}|+|y-y_{0}|\right)^{p}d\gamma_{X}d\gamma_{Y}.
\]
In turn, $\left(|x-x_{0}|+|y-y_{0}|\right)^{p}\leq2^{p-1}\left(|x-x_{0}|^{p}+|y-y_{0}|^{p}\right)$,
so 
\begin{align*}
\int_{\left(\mathbb{R}^{d_{X}}\right)^{2}}\int_{\left(\mathbb{R}^{d_{Y}}\right)^{2}}\left(|x-x_{0}|+|y-y_{0}|\right)^{p}d\gamma_{X}d\gamma_{Y} & \leq\int_{\left(\mathbb{R}^{d_{X}}\right)^{2}}\int_{\left(\mathbb{R}^{d_{Y}}\right)^{2}}2^{p-1}\left(|x-x_{0}|^{p}+|y-y_{0}|^{p}\right)d\gamma_{X}d\gamma_{Y}\\
 & =2^{p-1}\left(\int_{\left(\mathbb{R}^{d_{X}}\right)^{2}}|x-x_{0}|^{p}d\gamma_{X}+\int_{\left(\mathbb{R}^{d_{Y}}\right)^{2}}|y-y_{0}|^{p}d\gamma_{Y}\right)\\
 & =2^{p-1}\left(W_{p}^{p}\left(\mathcal{L}_{X|Z\in V_{j}},\mathcal{L}_{X|Z=z_{0}}\right)+W_{p}^{p}\left(\mathcal{L}_{Y|Z\in V_{j}},\mathcal{L}_{Y|Z=z_{0}}\right)\right).
\end{align*}
Since both $z\mapsto\mathcal{L}_{X|Z=z}$ and $z\mapsto\mathcal{L}_{Y|Z=z}$
are Lipschitz continuous (with respective Lipschitz constants $L_{X}$
and $L_{Y}$), we know that for $z,z_{0}\in V_{j}$, 
\[
W_{p}\left(\mathcal{L}_{X|Z=z},\mathcal{L}_{X|Z=z_{0}}\right)\leq L_{X}|z-z_{0}|\leq\frac{L_{X}}{\left(2^{p-1}\left(L_{X}^{p}+L_{Y}^{p}\right)\right)^{1/p}}\varepsilon
\]
and likewise
\[
W_{p}\left(\mathcal{L}_{Y|Z=z},\mathcal{L}_{Y|Z=z_{0}}\right)\leq L_{Y}|z-z_{0}|\leq\frac{L_{Y}}{\left(2^{p-1}\left(L_{X}^{p}+L_{Y}^{p}\right)\right)^{1/p}}\varepsilon.
\]
Therefore, we can invoke Lemma \ref{lem:disintegration} with $\mu_{z}=\mathcal{L}_{X|Z=z}$,
$\nu=\mathcal{L}_{X|Z=z_{0}},$ and conclude, as in the proof of the
previous proposition, that 
\begin{align*}
W_{p}^{p}\left(\mathcal{L}_{X|Z\in V_{j}},\mathcal{L}_{X|Z=z_{0}}\right) & \leq\fint_{V_{j}}W_{p}^{p}\left(\mathcal{L}_{X|Z=z},\mathcal{L}_{X|Z=z_{0}}\right)d\mathcal{L}_{Z}(z)\\
 & =\left(\frac{L_{X}}{\left(2^{p-1}\left(L_{X}^{p}+L_{Y}^{p}\right)\right)^{1/p}}\varepsilon\right)^{p}
\end{align*}
and, by identical reasoning,
\[
W_{p}^{p}\left(\mathcal{L}_{Y|Z\in V_{j}},\mathcal{L}_{Y|Z=z_{0}}\right)\leq\left(\frac{L_{Y}}{\left(2^{p-1}\left(L_{X}^{p}+L_{Y}^{p}\right)\right)^{1/p}}\varepsilon\right)^{p}.
\]
Therefore, 
\begin{align*}
W_{p}^{p}\left(\mathcal{L}_{X|Z\in V_{j}}\otimes\mathcal{L}_{Y|Z\in V_{j}},\mathcal{L}_{X|Z=z_{0}}\otimes\mathcal{L}_{Y|Z=z_{0}}\right) & \leq2^{p-1}\left(\frac{L_{X}}{\left(2^{p-1}\left(L_{X}^{p}+L_{Y}^{p}\right)\right)^{1/p}}\varepsilon\right)^{p}\\
 & \quad+2^{p-1}\left(\frac{L_{Y}}{\left(2^{p-1}\left(L_{X}^{p}+L_{Y}^{p}\right)\right)^{1/p}}\varepsilon\right)^{p}\\
 & =\varepsilon^{p}
\end{align*}
as desired.
\end{proof}
\begin{rem*}
Using essentially the same computation as in the proof above, one
can show that if $z\mapsto\mathcal{L}_{X|Z=z}$ and $z\mapsto\mathcal{L}_{Y|Z=z}$
are $L_{X}$- and $L_{Y}$-Lipschitz (resp.) with respect to $W_{p}$,
then $z\mapsto\mathcal{L}_{X|Z=z}\otimes\mathcal{L}_{Y|Z=z}$ is Lipschitz
with constant $\left(2^{p-1}\left(L_{X}^{p}+L_{Y}^{p}\right)\right)^{1/p}$.
Of course, under the null hypothesis, $\mathcal{L}_{X|Z=z}\otimes\mathcal{L}_{Y|Z=z}=\mathcal{L}_{(X,Y)|Z=z}$
for all $z$, so in particular it must be the case that $\left(2^{p-1}\left(L_{X}^{p}+L_{Y}^{p}\right)\right)^{1/p}=L_{XY}$.
Under the alternative, however, there is no \emph{a priori} relationship
between $L_{X}$, $L_{Y}$, and $L_{XY}$. 
\end{rem*}
{} %

\begin{cor}
\label{cor:binning error}Suppose that $z\mapsto\mathcal{L}_{(X,Y)|Z=z}$
be $L_{XY}$-Lipschitz with respect to $W_{p}$, and that $z\mapsto\mathcal{L}_{X|Z=z}$
and $z\mapsto\mathcal{L}_{Y|Z=z}$ are $L_{X}$- and $L_{Y}$-Lipschitz
(resp.) with respect to $W_{p}$. Suppose that $\{V_{j}\}_{j=1}^{J}$
is a measurable partition of the support of $\mathcal{L}_{Z}$, and
that for every bin $V_{j}$,
\[
diam(V_{j})\leq\frac{\varepsilon}{4\max\left\{ L_{XY},\left(2^{p-1}\left(L_{X}^{p}+L_{Y}^{p}\right)\right)^{1/p}\right\} }.
\]
Then, 
\[
\left|W_{p}\left(\mathcal{L}_{(X,Y)|Z=z_{0}},\mathcal{L}_{X|Z=z_{0}}\otimes\mathcal{L}_{Y|Z=z_{0}}\right)-W_{p}\left(\mathcal{L}_{(X,Y)|Z\in V_{j}},\mathcal{L}_{X|Z\in V_{j}}\otimes\mathcal{L}_{Y|Z\in V_{j}}\right)\right|\leq\frac{\varepsilon}{2}.
\]
\end{cor}

\begin{proof}
This follows directly from Propositions \ref{prop:binning-1} and
\ref{prop:binning-2}, together with the triangle inequality.
\end{proof}
The significance of this corollary is the following. If, for every
bin $V_{j}$, we perform a $W_{p}$ two-sample test on the measures
$\mathcal{L}_{(X,Y)|Z\in V_{j}}$ and $\mathcal{L}_{X|Z\in V_{j}}\otimes\mathcal{L}_{Y|Z\in V_{j}}$,
then (under our running assumptions) this amounts to a $W_{p}$ two-sample
test on the measures $\mathcal{L}_{(X,Y)|Z=z_{0}}$ and $\mathcal{L}_{X|Z=z_{0}}\otimes\mathcal{L}_{Y|Z=z_{0}}$,
for \emph{every} $z_{0}\in V_{j}$. Consequently, under our assumptions,
``$W_{p}$ conditional independence testing'' is reducible to a
large, but finite, number of $W_{p}$ two-sample tests. %

To put this more formally, we now state a ``metatheorem'', indicating
how to aggregate Wasserstein two-sample tests on the bins $V_{j}$,
with given Type I and Type II error, into a conditional independence
test.
\begin{thm}
\label{thm:metatheorem}Let $X:\Omega\rightarrow\mathbb{R}^{d_{X}}$,
$Y:\Omega\rightarrow\mathbb{R}^{d_{Y}}$, and $Z:\Omega\rightarrow\mathbb{R}^{d_{Y}}$
be random variables. Suppose that $z\mapsto\mathcal{L}_{(X,Y)|Z=z}$
is $L_{XY}$-Lipschitz with respect to $W_{p}$, and that $z\mapsto\mathcal{L}_{X|Z=z}$
and $z\mapsto\mathcal{L}_{Y|Z=z}$ are $L_{X}$- and $L_{Y}$-Lipschitz
(resp.) with respect to $W_{p}$. Suppose %
{} the bins $\{V_{j}\}_{j=1}^{J}$ form a measurable partition of the
support of $\mathcal{L}_{Z}$,%
{} and that for every bin $V_{j}$,
\[
diam(V_{j})\leq\frac{\varepsilon}{4\max\left\{ L_{XY},\left(2^{p-1}\left(L_{X}^{p}+L_{Y}^{p}\right)\right)^{1/p}\right\} }.
\]
Let $W_{p}\left(\widehat{\mathcal{L}_{(X,Y)|Z\in V_{j}}},\widehat{\mathcal{L}_{X|Z\in V_{j}}}\otimes\widehat{\mathcal{L}_{Y|Z\in V_{j}}}\right)$
denote the plug-in estimator for $W_{p}\left(\mathcal{L}_{(X,Y)|Z\in V_{j}},\mathcal{L}_{X|Z\in V_{j}}\otimes\mathcal{L}_{Y|Z\in V_{j}}\right)$,
and let $n_{j}$ denote the number of data points drawn from each
of $\mathcal{L}_{(X,Y)|Z\in V_{j}}$ and $\mathcal{L}_{X|Z\in V_{j}}\otimes\mathcal{L}_{Y|Z\in V_{j}}$.
Then:
\begin{enumerate}
\item (Control of Type I error) Suppose that for a given $j\in\{1,\ldots,J\}$,
$n_{j}$ is sufficiently large that for $p_{j}:\mathbb{N}\rightarrow[0,1]$,
\[
\mathbb{P}\left(W_{p}\left(\widehat{\mathcal{L}_{(X,Y)|Z\in V_{j}}},\widehat{\mathcal{L}_{X|Z\in V_{j}}}\otimes\widehat{\mathcal{L}_{Y|Z\in V_{j}}}\right)\geq W_{p}\left(\mathcal{L}_{(X,Y)|Z\in V_{j}},\mathcal{L}_{X|Z\in V_{j}}\otimes\mathcal{L}_{Y|Z\in V_{j}}\right)+\frac{\varepsilon}{2}\right)\leq p_{j}(n_{j})
\]
Then for all $z\in V_{j}$,
\[
\mathbb{P}\left(W_{p}\left(\widehat{\mathcal{L}_{(X,Y)|Z\in V_{j}}},\widehat{\mathcal{L}_{X|Z\in V_{j}}}\otimes\widehat{\mathcal{L}_{Y|Z\in V_{j}}}\right)\geq W_{p}\left(\mathcal{L}_{(X,Y)|Z=z},\mathcal{L}_{X|Z=z}\otimes\mathcal{L}_{Y|Z=z}\right)+\varepsilon\right)\leq p_{j}(n_{j})
\]
and: under the null hypothesis that for all $z$ in the support of
$\mathcal{L}_{Z}$, $W_{p}\left(\mathcal{L}_{(X,Y)|Z=z},\mathcal{L}_{X|Z=z}\otimes\mathcal{L}_{Y|Z=z}\right)=0$,
it holds that
\[
\mathbb{P}\left(\exists j\in\{1,\ldots,J\},W_{p}\left(\widehat{\mathcal{L}_{(X,Y)|Z\in V_{j}}},\widehat{\mathcal{L}_{X|Z\in V_{j}}}\otimes\widehat{\mathcal{L}_{Y|Z\in V_{j}}}\right)\geq\varepsilon\right)\leq\sum_{j=1}^{J}p_{j}(n_{j}).
\]
\item (Control of Type II error) Suppose that for a given $j\in\{1,\ldots,J\}$,
$n_{j}$ is sufficiently large that for $\varepsilon,\delta_{j}>0$
and $\alpha_{j}:\mathbb{R}_{+}\rightarrow[0,1]$, 
\begin{multline*}
\mathbb{P}\left(W_{p}\left(\widehat{\mathcal{L}_{(X,Y)|Z\in V_{j}}},\widehat{\mathcal{L}_{X|Z\in V_{j}}}\otimes\widehat{\mathcal{L}_{Y|Z\in V_{j}}}\right)<W_{p}\left(\mathcal{L}_{(X,Y)|Z\in V_{j}},\mathcal{L}_{X|Z\in V_{j}}\otimes\mathcal{L}_{Y|Z\in V_{j}}\right)-\delta_{j}+\frac{\varepsilon}{2}\right)\\
\leq\alpha_{j}\left(\delta_{j}-\frac{\varepsilon}{2},n_{j}\right).
\end{multline*}
Then for all $z\in V_{j}$, 
\begin{multline*}
\mathbb{P}\left(W_{p}\left(\widehat{\mathcal{L}_{(X,Y)|Z\in V_{j}}},\widehat{\mathcal{L}_{X|Z\in V_{j}}}\otimes\widehat{\mathcal{L}_{Y|Z\in V_{j}}}\right)<W_{p}\left(\mathcal{L}_{(X,Y)|Z=z},\mathcal{L}_{X|Z=z}\otimes\mathcal{L}_{Y|Z=z}\right)-\delta_{j}+\varepsilon\right)\\
\leq\alpha_{j}\left(\delta_{j}-\frac{\varepsilon}{2},n_{j}\right)
\end{multline*}
and: under the alternative hypothesis that for every $j\in\{1,\ldots,J\}$,
there exists a $z\in V_{j}$ such that \textbf{$W_{p}\left(\mathcal{L}_{(X,Y)|Z=z},\mathcal{L}_{X|Z=z}\otimes\mathcal{L}_{Y|Z=z}\right)\geq\delta_{j}+\varepsilon$},
it holds that 
\[
\mathbb{P}\left(\forall j\in\{1,\ldots,J\},W_{p}\left(\widehat{\mathcal{L}_{(X,Y)|Z\in V_{j}}},\widehat{\mathcal{L}_{X|Z\in V_{j}}}\otimes\widehat{\mathcal{L}_{Y|Z\in V_{j}}}\right)<\varepsilon\right)\leq\prod_{j=1}^{J}\alpha_{j}\left(\delta_{j}-\frac{\varepsilon}{2},n_{j}\right).
\]
\end{enumerate}
\end{thm}

\begin{proof}
(1) This is immediate from Corollary \ref{cor:binning error}, together
with a Bonferroni correction.

(2) This is also an immediate consequence of Corollary \ref{cor:binning error},
together with the fact that the bins $V_{j}$ are disjoint, and hence
the events 
\[
W_{p}\left(\widehat{\mathcal{L}_{(X,Y)|Z\in V_{j}}},\widehat{\mathcal{L}_{X|Z\in V_{j}}}\otimes\widehat{\mathcal{L}_{Y|Z\in V_{j}}}\right)<W_{p}\left(\mathcal{L}_{(X,Y)|Z\in V_{j}},\mathcal{L}_{X|Z\in V_{j}}\otimes\mathcal{L}_{Y|Z\in V_{j}}\right)-\delta_{j}+\frac{\varepsilon}{2},
\]
for $j\in\{1,\ldots,J\}$, are independent. Explicitly, %
{} independence implies that 
\begin{align*}
 & \mathbb{P}\left(\forall j\leq J,W_{p}\left(\widehat{\mathcal{L}_{(X,Y)|Z\in V_{j}}},\widehat{\mathcal{L}_{X|Z\in V_{j}}}\otimes\widehat{\mathcal{L}_{Y|Z\in V_{j}}}\right)<W_{p}\left(\mathcal{L}_{(X,Y)|Z\in V_{j}},\mathcal{L}_{X|Z\in V_{j}}\otimes\mathcal{L}_{Y|Z\in V_{j}}\right)-\delta_{j}+\frac{\varepsilon}{2}\right)\\
 & =\prod_{j=1}^{J}\mathbb{P}\left(W_{p}\left(\widehat{\mathcal{L}_{(X,Y)|Z\in V_{j}}},\widehat{\mathcal{L}_{X|Z\in V_{j}}}\otimes\widehat{\mathcal{L}_{Y|Z\in V_{j}}}\right)<W_{p}\left(\mathcal{L}_{(X,Y)|Z\in V_{j}},\mathcal{L}_{X|Z\in V_{j}}\otimes\mathcal{L}_{Y|Z\in V_{j}}\right)-\delta_{j}+\frac{\varepsilon}{2}\right)\\
 & \leq\prod_{j=1}^{J}\alpha_{j}\left(\delta_{j}-\frac{\varepsilon}{2},n_{j}\right)
\end{align*}
and, under the stipulated alternative, namely $W_{p}\left(\mathcal{L}_{(X,Y)|Z=z},\mathcal{L}_{X|Z=z}\otimes\mathcal{L}_{Y|Z=z}\right)\geq\delta_{j}+\varepsilon$
for some $z\in V_{j}$, we have by Corollary \ref{cor:binning error}
that $W_{p}\left(\mathcal{L}_{(X,Y)|Z\in V_{j}},\mathcal{L}_{X|Z\in V_{j}}\otimes\mathcal{L}_{Y|Z\in V_{j}}\right)\geq\delta_{j}+\frac{\varepsilon}{2}$,
and hence
\begin{multline*}
\mathbb{P}\left(\forall j\leq J,W_{p}\left(\widehat{\mathcal{L}_{(X,Y)|Z\in V_{j}}},\widehat{\mathcal{L}_{X|Z\in V_{j}}}\otimes\widehat{\mathcal{L}_{Y|Z\in V_{j}}}\right)<\varepsilon\right)\\
\leq\mathbb{P}\left(\forall j\leq J,W_{p}\left(\widehat{\mathcal{L}_{(X,Y)|Z\in V_{j}}},\widehat{\mathcal{L}_{X|Z\in V_{j}}}\otimes\widehat{\mathcal{L}_{Y|Z\in V_{j}}}\right)<W_{p}\left(\mathcal{L}_{(X,Y)|Z\in V_{j}},\mathcal{L}_{X|Z\in V_{j}}\otimes\mathcal{L}_{Y|Z\in V_{j}}\right)-\delta_{j}+\frac{\varepsilon}{2}\right).
\end{multline*}
Hence 
\[
\mathbb{P}\left(\forall j\in\{1,\ldots,J\},W_{p}\left(\widehat{\mathcal{L}_{(X,Y)|Z\in V_{j}}},\widehat{\mathcal{L}_{X|Z\in V_{j}}}\otimes\widehat{\mathcal{L}_{Y|Z\in V_{j}}}\right)<\varepsilon\right)\leq\prod_{j=1}^{J}\alpha_{j}\left(\delta_{j}-\frac{\varepsilon}{2},n_{j}\right)
\]
as desired.
\end{proof}
\begin{rem*}
Note that $J$, the number of bins partitioning $supp(\mathcal{L}_{Z})$,
is on the order of the\linebreak{}
 $\varepsilon/(4\max\{L_{XY},(2^{p-1}(L_{X}^{p}+L_{Y}^{p}))^{1/p}\})$-covering
number of $supp(\mathcal{L}_{Z})$, and so if $supp(\mathcal{L}_{Z})$
is geometrically ``simple'' (i.e. has small isoperimetric ratio
compared to a $d_{Z}$-ball), then $J$ is on the order of 
\[
\left(8\cdot diam(supp(\mathcal{L}_{Z}))\max\left\{ L_{XY},\left(2^{p-1}\left(L_{X}^{p}+L_{Y}^{p}\right)\right)^{1/p}\right\} \sqrt{d}/\varepsilon\right)^{d}.
\]
\end{rem*}
\begin{example}
\label{exa:d>=00003D3 example}Let us see how the preceding theorem
caches out in a more specific instance, where we have made more specific
regularity assumptions on the random variables $X$, $Y$, and $Z$,
and the distributions $z\mapsto\mathcal{L}_{(X,Y)|Z=z}$ and $z\mapsto\mathcal{L}_{X|Z=z}\otimes\mathcal{L}_{Y|Z=z}$.
Whatever the case, if we are to employ an expectation bound or concentration
inequality (or analogous substitutes) that require some additional
regularity assumptions, we must check that those regularity properties
are preserved when we pass from $z\mapsto\mathcal{L}_{(X,Y)|Z=z}$
and $z\mapsto\mathcal{L}_{X|Z=z}\otimes\mathcal{L}_{Y|Z=z}$ to the
binned distributions $j\mapsto\mathcal{L}_{(X,Y)|Z\in V_{j}}$ and
$j\mapsto\mathcal{L}_{X|Z\in V_{j}}\otimes\mathcal{L}_{Y|Z\in V_{j}}$
(since it is more natural to impose regularity assumptions on the
(unbinned) conditional distributions $\mathcal{L}_{(X,Y)|Z=z}$, $\mathcal{L}_{X|Z=z}$,
and $\mathcal{L}_{Y|Z=z}$). 

Here, consider the case where $X$ and $Y$ are both compactly supported
inside a ball of radius $D$ inside $\mathbb{R}^{d_{X}}\times\mathbb{R}^{d_{Y}}$,
and put $d=\min\{3,d_{X}+d_{Y}\}$. Therefore we have access to the
expectation bound of Dereich et al. stated in Theorem \ref{thm:expectation bound},
and the concentration inequality of Weed \& Bach stated in Proposition
\ref{prop:concentration}; let us check that the requisite assumptions
for these two bounds are stable under binning. 

Indeed, recall the expectation bound from Theorem \ref{thm:expectation bound},
namely, that for $d\geq3$, if $p\in[1,d/2)$ and $q>dp/(d-p)$, then
there exists a constant $\kappa_{p,q,d}$ depending only on $p,q,$
and $d$, such that 
\[
\mathbb{E}[W_{p}^{p}(\mu,\mu_{n})]^{1/p}\leq\kappa_{p,q,d}\left[\int_{\mathbb{R}^{d}}\Vert x\Vert^{q}d\mu(x)\right]^{1/q}n^{-1/d}.
\]
So suppose that there is some uniform $M_{XY}$ such that 
\[
\left[\int_{\mathbb{R}^{d_{X}+d_{Y}}}\Vert(x,y)\Vert^{q}d\mu(x,y)\right]^{1/q}\leq M_{XY}
\]
whenever $\mu=\mathcal{L}_{(X,Y)|Z=z}$ and $z\in supp(\mathcal{L}_{Z})$.
We claim that the same moment bound then holds for $\mathcal{L}_{(X,Y)|z\in V_{j}}$
also, at least provided that $z\mapsto\mathcal{L}_{(X,Y)|Z=z}$ is
continuous w.r.t. $W_{q}$ (in addition to being $L_{XY}$-Lipschitz
w.r.t. $W_{p}$; since convergence in $W_{q}$ is equivalent to weak
convergence plus convergence of $q$th moments, it suffices, for instance,
to assume yet another uniform moment bound, for any $q^{\prime}>q$:
see \cite[Corollary following Theorem 25.12]{billingsley1995probability}).
Indeed, 
\[
\int_{\mathbb{R}^{d_{X}+d_{Y}}}\Vert(x,y)\Vert^{q}d\mu(x,y)=W_{q}^{q}(\mu,\delta_{0})
\]
and so, by Lemma \ref{lem:disintegration}, 
\begin{align*}
\int_{\mathbb{R}^{d_{X}+d_{Y}}}\Vert(x,y)\Vert^{q}d\mathcal{L}_{(X,Y)|Z\in V_{j}}(x,y) & =W_{q}^{q}(\mathcal{L}_{(X,Y)|z\in V_{j}},\delta_{0})\\
 & \leq\fint_{V_{j}}W_{q}^{q}(\mathcal{L}_{(X,Y)|Z=z},\delta_{0})d\mathcal{L}_{Z}(z)\\
 & \leq\fint_{V_{j}}M_{XY}^{q}d\mathcal{L}_{Z}(z)\\
 & =M_{XY}^{q}.
\end{align*}
Likewise, it is obvious that if, for all $z\in supp\mathcal{L}_{Z}$,
$\mathcal{L}_{(X,Y)|Z=z}$ is supported inside some ball $B(0,D)$,
then the same is true for each $\mathcal{L}_{(X,Y)|Z\in V_{j}}$.
And identical reasoning applies to both $\mathcal{L}_{X|Z=z}$ and
$\mathcal{L}_{Y|Z=z}$. Finally, observe that 
\[
\int_{\mathbb{R}^{d_{X}+d_{Y}}}\Vert(x,y)\Vert^{q}d\mathcal{L}_{X|Z\in V_{j}}\otimes\mathcal{L}_{Y|Z\in V_{j}}(x,y)\leq2^{q-1}\left(\int_{\mathbb{R}^{d_{X}}}\Vert x\Vert^{q}d\mathcal{L}_{X|Z\in V_{j}}(x)+\int_{\mathbb{R}^{d_{Y}}}\Vert y\Vert^{q}d\mathcal{L}_{Y|Z\in V_{j}}(y)\right)
\]
so, if $\mathcal{L}_{X|Z\in V_{j}}$ has $q$th moment bound $M_{X}$
and $\mathcal{L}_{Y|Z\in V_{j}}$ has $q$th moment bound $M_{Y}$,
it follows that $\mathcal{L}_{X|Z\in V_{j}}\otimes\mathcal{L}_{Y|Z\in V_{j}}$
has $q$th moment bound $\left(2^{q-1}\left(M_{X}^{q}+M_{Y}^{q}\right)\right)^{1/q}$. 

Thus, if we assume that $z\mapsto\mathcal{L}_{(X,Y)|Z=z}$ and $z\mapsto\mathcal{L}_{X|Z=z}$
and $z\mapsto\mathcal{L}_{Y|Z=z}$ are:
\begin{enumerate}
\item $L_{XY}$-, $L_{X}$-, and $L_{Y}$-Lipschitz with respect to $W_{p}$
for some $p\in[1,\infty)$, 
\item continuous with respect to $W_{q}$ for some $q>dp/(d-p)$, where
$d=\min\{3,d_{X}+d_{Y}\}$, and
\item have $q$th moments uniformly bounded by $M_{XY}$ and $M_{X}$ and
$M_{Y}$ respectively, and
\item have bounded support\footnote{Note: this implies automatically that (3) holds, but we run the analysis
with $M_{X}$, $M_{Y}$, $M_{XY}$ given more generally, in case one
has access to tighter moment bounds than those implied automatically
from the diameter of the support.} with radius $D$, 
\end{enumerate}
Then we can use both the expectation bound stated in Theorem \ref{thm:expectation bound},
as well as the concentration inequality stated in Proposition \ref{prop:concentration},
in the case of the empirical measures of $\mathcal{L}_{(X,Y)|Z\in V_{j}}$
and $\mathcal{L}_{X|Z\in V_{j}}\otimes\mathcal{L}_{Y|Z\in V_{j}}$.

Indeed,  we saw in the previous section that under the null hypothesis,
if $\mu_{n}$ and $\mu_{n}^{\prime}$ are both samples from the same
measure $\mu$, and $n$ is sufficiently large that $\left(\kappa_{p,q,d}\left[\int_{\mathbb{R}^{d}}\Vert x\Vert^{q}d\mu(x)\right]^{1/q}n^{-1/d}\right)^{p}\leq\frac{1}{2}\left(\frac{\varepsilon}{4}\right)^{p}$,
then combining the expectation bound from Theorem \ref{thm:expectation bound}
and the concentration inequality from Proposition \ref{prop:concentration}
results in the estimate
\begin{align*}
\mathbb{P}\left[W_{p}(\mu_{n},\mu_{n}^{\prime})\geq\frac{\varepsilon}{2}\right] & \leq\mathbb{P}\left[W_{p}(\mu_{n},\mu)\geq\frac{\varepsilon}{4}\right]+\mathbb{P}\left[W_{p}(\mu_{n}^{\prime},\mu)\geq\frac{\varepsilon}{4}\right]\\
 & \leq2\exp\left(-\frac{2n\varepsilon^{2p}}{4^{2p+1}D^{2p}}\right).
\end{align*}
So, by the preceding theorem, it holds under the null that if for
all $j\in\{1,\ldots,J\}$, $n_{j}$ is sufficiently large that (again
$d=\min\{3,d_{X}+d_{Y}\}$)
\[
\left(\kappa_{p,q,d}\left[\int_{\mathbb{R}^{d_{X}+d_{Y}}}\Vert(x,y)\Vert^{q}d\mathcal{L}_{(X,Y)|Z\in V_{j}}(x,y)\right]^{1/q}n_{j}^{-1/d}\right)^{p}\leq\left(\kappa_{p,q,d}M_{XY}n_{j}^{-1/d}\right)^{p}\leq\frac{1}{2}\left(\frac{\varepsilon}{4}\right)^{p},
\]
it holds that 
\begin{align*}
\mathbb{P}\left[(\exists j\in\{1,\ldots J\})W_{p}(\widehat{\mathcal{L}_{(X,Y)|Z\in V_{j}}},\widehat{\mathcal{L}_{X|Z\in V_{j}}}\otimes\widehat{\mathcal{L}_{Y|Z\in V_{j}}})\geq\varepsilon\right] & \leq\sum_{j=1}^{J}2\exp\left(-\frac{2n_{j}\varepsilon^{2p}}{4^{2p+1}D^{2p}}\right)
\end{align*}
which is precisely a bound on the Type I error. 

Let us now turn to the Type II error. As discussed in the previous
section, we reason as follows. Our alternative hypothesis is that
for each $j\in\{1,\ldots J\}$, there exists a $z\in V_{j}$ such
that \textbf{\textit{$W_{p}\left(\mathcal{L}_{(X,Y)|Z=z},\mathcal{L}_{X|Z=z}\otimes\mathcal{L}_{Y|Z=z}\right)\geq\delta_{j}+\varepsilon$}};
hence $W_{p}\left(\mathcal{L}_{(X,Y)|Z\in V_{j}},\mathcal{L}_{X|Z\in V_{j}}\otimes\mathcal{L}_{Y|Z\in V_{j}}\right)\geq\delta_{j}+\frac{\varepsilon}{2}$.
If 
\[
W_{p}\left(\widehat{\mathcal{L}_{(X,Y)|Z\in V_{j}}},\widehat{\mathcal{L}_{X|Z\in V_{j}}}\otimes\widehat{\mathcal{L}_{Y|Z\in V_{j}}}\right)<W_{p}\left(\mathcal{L}_{(X,Y)|Z\in V_{j}},\mathcal{L}_{X|Z\in V_{j}}\otimes\mathcal{L}_{Y|Z\in V_{j}}\right)-\delta_{j}+\frac{\varepsilon}{2}
\]
then it must be the case that 
\[
W_{p}\left(\widehat{\mathcal{L}_{(X,Y)|Z\in V_{j}}},\mathcal{L}_{(X,Y)|Z\in V_{j}}\right)\geq\frac{\delta_{j}}{2}-\frac{\varepsilon}{4}\text{ or }W_{p}\left(\widehat{\mathcal{L}_{X|Z\in V_{j}}}\otimes\widehat{\mathcal{L}_{Y|Z\in V_{j}}},\mathcal{L}_{X|Z\in V_{j}}\otimes\mathcal{L}_{Y|Z\in V_{j}}\right)\geq\frac{\delta_{j}}{2}-\frac{\varepsilon}{4}.
\]
In the first case, using $t=\frac{1}{2}\left(\frac{\delta_{j}-\frac{\varepsilon}{2}}{4}\right)^{p}$
in Proposition \ref{prop:concentration}, we see that 
\begin{align*}
\mathbb{P}\left[W_{p}\left(\widehat{\mathcal{L}_{(X,Y)|Z\in V_{j}}},\mathcal{L}_{(X,Y)|Z\in V_{j}}\right)\geq\frac{\delta_{j}}{2}-\frac{\varepsilon}{4}\right] & \leq\exp\left(-\frac{2n_{j}\left(\delta_{j}-\frac{\varepsilon}{2}\right)^{2p}}{4^{2p+1}D^{2p}}\right)
\end{align*}
provided that $n_{j}$ is sufficiently large that 
\[
\left(\kappa_{p,q,d}\left[\int_{\mathbb{R}^{d_{X}}\times\mathbb{R}^{d_{Y}}}\Vert(x,y)\Vert^{q}d\mathcal{L}_{(X,Y)|Z\in V_{j}}(x,y)\right]^{1/q}n^{-1/d}\right)^{p}\leq\frac{1}{2}\left(\frac{\delta_{j}-\frac{\varepsilon}{2}}{4}\right)^{p}
\]
or, in terms of the moment bound $M_{XY}$, that 
\[
\left(\kappa_{p,q,d}M_{XY}n^{-1/d}\right)^{p}\leq\frac{1}{2}\left(\frac{\delta_{j}-\frac{\varepsilon}{2}}{4}\right)^{p}.
\]
In the second case, similarly,
\[
\mathbb{P}\left[W_{p}\left(\widehat{\mathcal{L}_{X|Z\in V_{j}}}\otimes\widehat{\mathcal{L}_{Y|Z\in V_{j}}},\mathcal{L}_{X|Z\in V_{j}}\otimes\mathcal{L}_{Y|Z\in V_{j}}\right)\geq\frac{\delta_{j}}{2}-\frac{\varepsilon}{4}\right]\leq\exp\left(-\frac{2n_{j}\left(\delta_{j}-\frac{\varepsilon}{2}\right)^{2p}}{4^{2p+1}D^{2p}}\right)
\]
provided that $n_{j}$ is sufficiently large that 
\[
\left(\kappa_{p,q,d}\left[\int_{\mathbb{R}^{d_{X}}\times\mathbb{R}^{d_{Y}}}\Vert(x,y)\Vert^{q}d\mathcal{L}_{X|Z\in V_{j}}\otimes\mathcal{L}_{Y|Z\in V_{j}}(x,y)\right]^{1/q}n^{-1/d}\right)^{p}\leq\frac{1}{2}\left(\frac{\delta_{j}-\frac{\varepsilon}{2}}{4}\right)^{p}
\]
or, in terms of our moment bound, that 
\[
\left(\kappa_{p,q,d}\left(2^{q-1}(M_{X}^{q}+M_{Y}^{q})\right)^{1/q}n^{-1/d}\right)^{p}\leq\frac{1}{2}\left(\frac{\delta_{j}-\frac{\varepsilon}{2}}{4}\right)^{p}.
\]
 Therefore, 
\begin{multline*}
\mathbb{P}\left[W_{p}\left(\widehat{\mathcal{L}_{(X,Y)|Z\in V_{j}}},\widehat{\mathcal{L}_{X|Z\in V_{j}}}\otimes\widehat{\mathcal{L}_{Y|Z\in V_{j}}}\right)\leq W_{p}\left(\mathcal{L}_{(X,Y)|Z\in V_{j}},\mathcal{L}_{X|Z\in V_{j}}\otimes\mathcal{L}_{Y|Z\in V_{j}}\right)-\delta_{j}+\frac{\varepsilon}{2}\right]\\
\leq2\exp\left(-\frac{2n_{j}\left(\delta_{j}-\frac{\varepsilon}{2}\right)^{2p}}{4^{2p+1}D^{2p}}\right)
\end{multline*}
provided that $n_{j}$ is sufficiently large that (again, $d=\min\{3,d_{X}+d_{Y}\}$)
\[
\left(\max\left\{ M_{XY},\left(2^{q-1}(M_{X}^{q}+M_{Y}^{q})\right)^{1/q}\right\} \kappa_{p,q,d}n_{j}^{-1/d}\right)^{p}\leq\frac{1}{2}\left(\frac{\delta_{j}-\frac{\varepsilon}{2}}{4}\right)^{p}.
\]
Combining these estimates across bins, we conclude that 
\begin{align*}
\mathbb{P}\left(\forall j\in\{1,\ldots,J\},W_{p}\left(\widehat{\mathcal{L}_{(X,Y)|Z\in V_{j}}},\widehat{\mathcal{L}_{X|Z\in V_{j}}}\otimes\widehat{\mathcal{L}_{Y|Z\in V_{j}}}\right)<\varepsilon\right) & \leq\prod_{j=1}^{J}2\exp\left(-\frac{2n_{j}\left(\delta_{j}-\frac{\varepsilon}{2}\right)^{2p}}{4^{2p+1}D^{2p}}\right)
\end{align*}
provided that for all $j\in\{1,\ldots,J\}$, $\left(\max\left\{ M_{XY},\left(2^{q-1}(M_{X}^{q}+M_{Y}^{q})\right)^{1/q}\right\} \kappa_{p,q,d}n_{j}^{-1/d}\right)^{p}\leq\frac{1}{2}\left(\frac{\delta_{j}-\frac{\varepsilon}{2}}{4}\right)^{p}$. 
\end{example}

\begin{rem*}
Using the concentration inequality from Proposition \ref{prop:unbounded Wp concentration ineq},
it is also possible to extend the analysis from the previous example
to the case where $\mathcal{L}_{X}$ and $\mathcal{L}_{Y}$ have unbounded
support (but still satisfy a higher moment bound). %
\end{rem*}

\section{Plug-in estimation of the Lipschitz constant\label{sec:Plug-in-estimation}}

We have seen that the role played by the Lipschitz constants $L_{X}$,
$L_{Y}$, and $L_{XY}$ is that they specify a bin size which is sufficient
to introduce only a small discretization error. In specific applications,
there may well be domain knowledge that allows us to upper bound the
$L$'s in an \emph{a priori} fashion; however, testing whether a function
(with continuous domain) is Lipschitz, or testing for $L$ given the
fact that the function is Lipschitz with \emph{some} constant, is
not possible in general, at least if one requires any sort of finite-sample
guarantees. (To see why this is the case, it suffices to consider
some standard pathological function from real analysis, such as the
Weierstrass function.) 
\begin{rem*}
In a situation similar to the one from Example \ref{exa:additive noise model}
above, some \emph{a priori}/expert knowledge about $L$ may come from
the functional form of a causal model under investigation, e.g. if
one considers an additive noise model as we did there. Additionally,
it may be possible to deduce a upper bound on $L$ analytically, for
instance if the conditional distributions are known to satisfy a suitable
\emph{functional inequality} (e.g. of Poincaré or log-Sobolev type)
--- for such a result, we refer the reader to \cite[Theorem 2.1]{dolera2020lipschitz}.
\end{rem*}
However, it is still of interest to provide a consistent \emph{estimator}
of $L$, should no guess based on domain knowledge be available. 

To that end, we note that it \emph{is }possible to compute the Lipschitz
constant (w.r.t. the $W_{p}$ metric) of the \emph{discrete} function
\[
V_{j}\mapsto\widehat{\mathcal{L}_{(X,Y)|Z\in V_{j}}}:=\frac{1}{|\{i:z_{i}\in V_{j}\}|}\sum_{i:z_{i}\in V_{j}}\delta_{(x_{i},y_{i})}
\]
which, given a bin, returns the empirical conditional distribution
for that bin. Briefly, this is because the space of bins $V_{j}$
can be viewed as a discrete space, and in such a setting one can simply
compute the Lipschitz constant by brute force. (Note, however, that
there is work on \emph{efficient} estimation of Lipschitz constants
for discrete functions, for instance \cite{awasthi2016testing}.) 

Therefore, we propose using the ``plug-in estimator'' $\widehat{L_{XY}}$
for the Lipschitz constant $L_{XY}$, namely 
\[
\widehat{L_{XY}}:=Lip\left(V_{j}\mapsto\widehat{\mathcal{L}_{(X,Y)|Z\in V_{j}}}\right)
\]
where $\widehat{\mathcal{L}_{(X,Y)|Z\in V_{j}}}$ resides in a Wasserstein
space $W_{p}(\mathbb{R}^{d_{X}+d_{Y}})$, and the distance between
two bins $V_{j}$ and $V_{j^{\prime}}$ (which we denote by $dist(V_{j},V_{j^{\prime}})$)
is given by the Euclidean distance between their centroids. In exactly
the same fashion, we also define the plug-in estimators $\widehat{L_{X}}$
and $\widehat{L_{Y}}$ for $L_{X}$ and $L_{Y}$, respectively, by
\[
\widehat{L_{X}}:=Lip\left(V_{j}\mapsto\widehat{\mathcal{L}_{X|Z\in V_{j}}}\right);
\]
\[
\widehat{L_{Y}}:=Lip\left(V_{j}\mapsto\widehat{\mathcal{L}_{Y|Z\in V_{j}}}\right).
\]

We now claim the following result, which requires that geometry of
the bins is in some sense ``uniform'', and the number of samples
per bin grows fast enough as the size of the bins goes to zero. (We
remark however that our assumptions allow for bins which are non-convex
or even disconnected.)
\begin{prop}
Suppose that the joint distribution $\mathcal{L}_{(X,Y,Z)}$ is compactly
supported inside $\mathbb{R}^{d_{X}+d_{Y}+d_{Z}}$, and that $z\mapsto\mathcal{L}_{(X,Y)|Z=z}$
is $L_{XY}$-Lipschitz, and that we have drawn $n$ i.i.d. samples
from $\mathcal{L}_{(X,Y,Z)}$. Suppose we have a set of bins $\{V_{j}\}_{j=1}^{J}$
whose ``scale'' depends on a paremeter $\varepsilon>0$ in a sense
which we specify below. Fix the following assumptions:
\begin{enumerate}
\item For all $j\in\{1,\ldots,J\}$, $\text{diam}(V_{j})\leq\varepsilon/L_{XY}$.
\item For all $j\in\{1,\ldots,J\}$, $c_{Vol}\leq\frac{\max_{j\in\{1,\ldots,J\}}Vol(V_{j})}{\min_{j\in\{1,\ldots,J\}}Vol(V_{j})}\leq C_{Vol}$,
where again $c_{Vol}$ and $C_{Vol}$ are independent of $\varepsilon$. 
\item Let $n_{min}$ denote the least number of samples from $\mathcal{L}_{(X,Y,Z)}$
belonging to the same bin $V_{j}$ (out of all $j\in\{1,\ldots,J\}$).
Assume that as $\varepsilon\rightarrow0$, $\varepsilon^{2d}\exp\left(-\frac{n_{min}\varepsilon^{2p}}{D^{2p}}\right)\rightarrow0$,
where $D$ is the diameter of the support of $\mathcal{L}_{(X,Y)}$;
and,
\item For some $q>p$, assume $\kappa_{q,p,d}M_{q}^{1/q}n_{min}^{-1/d}\leq\varepsilon$,
where $M_{q}$ is a uniform $q$th moment bound on $\mathcal{L}_{(X,Y)|Z=z}$,
$d=\min\{3,d_{X}+d_{Y}\}$, and $\kappa_{q,p,d}$ is the constant
from Theorem \ref{thm:expectation bound}.
\end{enumerate}
Then, $\widehat{L_{XY}}$ is a consistent estimator, in the following
sense: with probability approaching 1 as $\varepsilon\rightarrow0$,
it holds both that for all $j,j^{\prime}\in\{1,\ldots,J\}$,
\[
W_{p}(\widehat{\mathcal{L}_{(X,Y)|Z\in V_{j}}},\widehat{\mathcal{L}_{(X,Y)|Z\in V_{j^{\prime}}}})\leq L_{XY}\cdot dist(V_{j},V_{j^{\prime}})+(2+2^{1+1/p})\varepsilon
\]
and for all $z,z^{\prime}\in\text{supp}(\mathcal{L}_{Z})$, 
\[
W_{p}(\mathcal{L}_{(X,Y)|Z=z},\mathcal{L}_{(X,Y)|Z=z^{\prime}})\leq\widehat{L_{XY}}|z-z^{\prime}|+\left(2+2^{1+1/p}+2\frac{\widehat{L_{XY}}}{L_{XY}}\right)\varepsilon.
\]
\[
\]
 The same holds true for $\widehat{L_{X}}$ and $\widehat{L_{Y}}$. 
\end{prop}

\begin{rem*}
This proposition establishes that the plug-in estimator for the Lipschitz
constant is accurate with high probability, provided that we restrict
ourselves to considering points $z,z^{\prime}$ (resp. bins $V_{j},V_{j^{\prime}}$)
which are a macroscopic distance apart as compared to the length scale
parameter $\varepsilon$. This type of restriction on the result is
necessary, as binning the $Z$ variable can, \emph{a priori}, smooth
out oscillations at small length scales.
\end{rem*}
\begin{proof}
The proof is identical for all three estimators; we therefore proceed
only for the estimator $\widehat{L_{XY}}$. %

Let $\varepsilon>0$. Let $\{V_{j}\}_{j=1}^{J}$ be a measurable partition
of $\text{supp}(\mathcal{L}_{Z})$, which we take to depend on $\varepsilon$,
such that for all $j\in\{1,\ldots,J\}$, $\text{diam}(V_{j})\leq\varepsilon/L_{XY}$.
Furthermore, we require that there exist constants $c$ and $C$ which
are independent of $\varepsilon$, such that 
\[
c\leq\frac{\max_{j\in\{1,\ldots,J\}}Vol(V_{j})}{\min_{j\in\{1,\ldots,J\}}Vol(V_{j})}\leq C.
\]
Given $n$ samples from the joint distribution $\mathcal{L}_{(X,Y,Z)}$,
we let $n_{j}$ denote the number of samples for which $Z\in V_{j}$.
Likewise we write $n_{min}=\min_{j\in\{1,\ldots,J\}}n_{j}$. We note
that as $\varepsilon\rightarrow0$, $J\rightarrow\infty$; likewise,
the $n_{j}$'s (and hence $n_{min}$) increase at a rate that depends
on $\varepsilon$. 

Previously, in Proposition \ref{prop:binning-1}, we have seen that
if $\text{diam}(V_{j})\leq\varepsilon/L_{XY}$, then
\[
W_{p}\left(\mathcal{L}_{(X,Y)|Z\in V_{j}},\mathcal{L}_{(X,Y)|Z=z_{0}}\right)\leq\varepsilon.
\]
At the same time, for $n_{j}$ sufficiently large, we have that $\mathcal{L}_{(X,Y)|Z\in V_{j}}\approx_{W_{p}}\widehat{\mathcal{L}_{(X,Y)|Z\in V_{j}}}$,
with probability going to 1 as $n_{j}$ goes to $\infty$, where $n_{j}$
is the number of samples in the $j$th bin $V_{j}$. 

Therefore, we proceed as follows. Pick $z_{0}$ and $z_{0}^{\prime}$
such that $z_{0}\in V_{j}$ and $z_{0}^{\prime}\in V_{j^{\prime}}$.
We know that 
\[
W_{p}(\mathcal{L}_{(X,Y)|Z=z_{0}},\mathcal{L}_{(X,Y)|Z=z_{0}^{\prime}})\leq L_{XY}|z_{0}-z_{0}^{\prime}|.
\]
At the same time, by the triangle inequality and Proposition \ref{prop:binning-1},
we have (still with $z_{0}\in V_{j}$ and $z_{0}^{\prime}\in V_{j^{\prime}}$)
that
\[
W_{p}(\mathcal{L}_{(X,Y)|Z\in V_{j}},\mathcal{L}_{(X,Y)|Z\in V_{j^{\prime}}})\leq W_{p}(\mathcal{L}_{(X,Y)|Z=z_{0}},\mathcal{L}_{(X,Y)|Z=z_{0}^{\prime}})+2\varepsilon.
\]
Hence in particular,
\[
W_{p}(\mathcal{L}_{(X,Y)|Z\in V_{j}},\mathcal{L}_{(X,Y)|Z\in V_{j^{\prime}}})\leq L_{XY}|z_{0}-z_{0}^{\prime}|+2\varepsilon.
\]
Using the triangle inequality again, and the fact that $\mathcal{L}_{(X,Y)|Z\in V_{j}}\approx_{W_{p}}\widehat{\mathcal{L}_{(X,Y)|Z\in V_{j}}}$
and $\mathcal{L}_{(X,Y)|Z\in V_{j^{\prime}}}\approx_{W_{p}}\widehat{\mathcal{L}_{(X,Y)|Z\in V_{j^{\prime}}}}$
for $n_{j}$, $n_{j}^{\prime}$ very large (with probability approaching
1 as $n_{j},n_{j}^{\prime}\rightarrow\infty$), we have that (also
with probability approaching 1 as $n_{j},n_{j}^{\prime}\rightarrow\infty$)
\[
W_{p}(\widehat{\mathcal{L}_{(X,Y)|Z\in V_{j}}},\widehat{\mathcal{L}_{(X,Y)|Z\in V_{j^{\prime}}}})\lessapprox L_{XY}|z_{0}-z_{0}^{\prime}|+2\varepsilon
\]
where the symbol $\lessapprox$ conceals only an error term arising
from the sample complexity of $\widehat{\mathcal{L}_{(X,Y)|Z\in V_{j}}}$
and $\widehat{\mathcal{L}_{(X,Y)|Z\in V_{j^{\prime}}}}$. More explictly,
using Proposition \ref{prop:concentration}, we have that for each
$j$, 
\[
\mathbb{P}\left[W_{p}^{p}\left(\widehat{\mathcal{L}_{(X,Y)|Z\in V_{j}}},\mathcal{L}_{(X,Y)|Z\in V_{j}}\right)\geq\mathbb{E}W_{p}^{p}\left(\widehat{\mathcal{L}_{(X,Y)|Z\in V_{j}}},\mathcal{L}_{(X,Y)|Z\in V_{j}}\right)+\varepsilon^{p}\right]\leq\exp\left(-\frac{n_{j}\varepsilon^{2p}}{D^{2p}}\right)
\]
where $D$ is the diameter of the support of $\mathcal{L}(X,Y)$.
Note that using the inequality $2^{1-p}(x+y)^{p}\leq x^{p}+y^{p}$,
and Jensen's inequality applied to the expectation, this implies 
\[
\mathbb{P}\left[W_{p}\left(\widehat{\mathcal{L}_{(X,Y)|Z\in V_{j}}},\mathcal{L}_{(X,Y)|Z\in V_{j}}\right)\geq2^{(1-p)/p}\left(\mathbb{E}W_{p}\left(\widehat{\mathcal{L}_{(X,Y)|Z\in V_{j}}},\mathcal{L}_{(X,Y)|Z\in V_{j}}\right)+\varepsilon\right)\right]\leq\exp\left(-\frac{n_{j}\varepsilon^{2p}}{D^{2p}}\right).
\]
This shows that if $\mathbb{E}W_{p}\left(\widehat{\mathcal{L}_{(X,Y)|Z\in V_{j}}},\mathcal{L}_{(X,Y)|Z\in V_{j}}\right)\leq\varepsilon$
(which holds provided that $\kappa_{q,p,d}M_{q}^{1/q}n_{j}^{-1/d}\leq\varepsilon$,
as per Theorem \ref{thm:expectation bound}) then we have
\[
\mathbb{P}\left[W_{p}\left(\widehat{\mathcal{L}_{(X,Y)|Z\in V_{j}}},\mathcal{L}_{(X,Y)|Z\in V_{j}}\right)\geq2^{1/p}\varepsilon\right]\leq\exp\left(-\frac{n_{j}\varepsilon^{2p}}{D^{2p}}\right).
\]
Then, from the triangle inequality, 
\[
W_{p}(\widehat{\mathcal{L}_{(X,Y)|Z\in V_{j}}},\widehat{\mathcal{L}_{(X,Y)|Z\in V_{j^{\prime}}}})\leq L_{XY}|z_{0}-z_{0}^{\prime}|+(2+2^{1+1/p})\varepsilon
\]
with probability at least $1-\exp\left(-\frac{n_{j}\varepsilon^{2p}}{D^{2p}}\right)-\exp\left(-\frac{n_{j^{\prime}}\varepsilon^{2p}}{D^{2p}}\right)$. 

Now, this analysis holds for any $z_{0}$ and $z_{0}^{\prime}$; in
particular we may take $z_{0}$ and $z_{0}^{\prime}$ to be the centroids
of $V_{j}$ and $V_{j^{\prime}}$ (although note this may require
us to take a smaller $\varepsilon$). In this case, the ``distance''
between $V_{j}$ and $V_{j^{\prime}}$ that we have specified, and
denote by $dist(V_{j},V_{j^{\prime}})$, is none other than $|z_{0}-z_{0}^{\prime}|$.

Then, quantifying over all pairs of indices $j,j^{\prime}$, we see
that under the same assumptions on $\varepsilon$, it holds that 
\[
W_{p}(\widehat{\mathcal{L}_{(X,Y)|Z\in V_{j}}},\widehat{\mathcal{L}_{(X,Y)|Z\in V_{j^{\prime}}}})\leq L_{XY}\cdot dist(V_{j},V_{j^{\prime}})+(2+2^{1+1/p})\varepsilon
\]
with probability at least 
\[
1-\underset{j\neq j^{\prime}}{\sum_{j,j^{\prime}=1}^{J}}\left(\exp\left(-\frac{n_{j}\varepsilon^{2p}}{D^{2p}}\right)+\exp\left(-\frac{n_{j^{\prime}}\varepsilon^{2p}}{D^{2p}}\right)\right).
\]
Note that the sum has $\frac{1}{2}J(J-1)$ terms. Now, under the assumption
that all of the cells $V_{j}$ have volume within a fixed constant
multiple of each other (that is, independent of $\varepsilon$), it
holds that for some uniform constant $C$, $J\leq C\varepsilon^{d}$.
This implies that (changing the constant $C$ as necessary)
\[
1-\underset{j\neq j^{\prime}}{\sum_{j,j^{\prime}=1}^{J}}\left(\exp\left(-\frac{n_{j}\varepsilon^{2p}}{D^{2p}}\right)+\exp\left(-\frac{n_{j^{\prime}}\varepsilon^{2p}}{D^{2p}}\right)\right)\geq1-C\varepsilon^{2d}\exp\left(-\frac{n_{min}\varepsilon^{2p}}{D^{2p}}\right).
\]
We therefore require that $n_{min}$ grows fast enough, as a function
of $\varepsilon$, that $\varepsilon^{2d}\exp\left(-\frac{n_{min}\varepsilon^{2p}}{D^{2p}}\right)\rightarrow0$.
For, given this, we have that for all $j,j^{\prime}$,
\[
W_{p}(\widehat{\mathcal{L}_{(X,Y)|Z\in V_{j}}},\widehat{\mathcal{L}_{(X,Y)|Z\in V_{j^{\prime}}}})\leq L_{XY}\cdot dist(V_{j},V_{j^{\prime}})+(2+2^{1+1/p})\varepsilon
\]
with probability at least $1-C\varepsilon^{2d}\exp\left(-\frac{n_{min}\varepsilon^{2p}}{D^{2p}}\right)$. 

It remains, therefore, to show the other inequality. The argument
is very similar. Take $V_{j}$ and $V_{j^{\prime}}$ to be arbitrary.
We know that 
\[
W_{p}(\widehat{\mathcal{L}_{(X,Y)|Z\in V_{j}}},\widehat{\mathcal{L}_{(X,Y)|Z\in V_{j^{\prime}}}})\leq\widehat{L_{XY}}\cdot dist(V_{j},V_{j^{\prime}})
\]
simply from the definition of $\hat{L}$. For a very large number
of data points $n_{j}$ and $n_{j^{\prime}}$, we have that $\mathcal{L}_{(X,Y)|Z\in V_{j}}\approx_{W_{p}}\widehat{\mathcal{L}_{(X,Y)|Z\in V_{j}}}$
and $\mathcal{L}_{(X,Y)|Z\in V_{j^{\prime}}}\approx_{W_{p}}\widehat{\mathcal{L}_{(X,Y)|Z\in V_{j^{\prime}}}}$
(with probability approaching 1 as $n_{j},n_{j}^{\prime}\rightarrow\infty$),
and so
\[
W_{p}(\mathcal{L}_{(X,Y)|Z\in V_{j}},\mathcal{L}_{(X,Y)|Z\in V_{j^{\prime}}})\lessapprox\widehat{L_{XY}}\cdot dist(V_{j},V_{j^{\prime}}).
\]
More explicitly, if $\mathbb{E}W_{p}\left(\widehat{\mathcal{L}_{(X,Y)|Z\in V_{j}}},\mathcal{L}_{(X,Y)|Z\in V_{j}}\right)\leq\varepsilon$
(which holds provided that $\kappa_{q,p,d}M_{q}^{1/q}n_{j}^{-1/d}\leq\varepsilon$,
as per Theorem \ref{thm:expectation bound}) then we have
\[
\mathbb{P}\left[W_{p}\left(\widehat{\mathcal{L}_{(X,Y)|Z\in V_{j}}},\mathcal{L}_{(X,Y)|Z\in V_{j}}\right)\geq2^{1/p}\varepsilon\right]\leq\exp\left(-\frac{n_{j}\varepsilon^{2p}}{D^{2p}}\right).
\]
Then, from the triangle inequality, 
\[
W_{p}(\mathcal{L}_{(X,Y)|Z\in V_{j}},\mathcal{L}_{(X,Y)|Z\in V_{j^{\prime}}})\leq\widehat{L_{XY}}\cdot dist(V_{j},V_{j^{\prime}})+2^{1+1/p}\varepsilon
\]
with probability at least $1-\exp\left(-\frac{n_{j}\varepsilon^{2p}}{D^{2p}}\right)-\exp\left(-\frac{n_{j^{\prime}}\varepsilon^{2p}}{D^{2p}}\right)$.%

Let $z_{0}$ and $z_{0}^{\prime}$ be the centroids of $V_{j}$ and
$V_{j^{\prime}}$. Then (still provided that $\kappa_{q,p,d}M_{q}^{1/q}n_{j}^{-1/d}\leq\varepsilon$,
and with probability at least $1-\exp\left(-\frac{n_{j}\varepsilon^{2p}}{D^{2p}}\right)-\exp\left(-\frac{n_{j^{\prime}}\varepsilon^{2p}}{D^{2p}}\right)$)
\[
W_{p}(\mathcal{L}_{(X,Y)|Z\in V_{j}},\mathcal{L}_{(X,Y)|Z\in V_{j^{\prime}}}))\leq\widehat{L_{XY}}|z_{0}-z_{0}^{\prime}|+2^{1+1/p}\varepsilon.
\]
From the triangle inequality and the binning consistency estimate,
we have that %
\[
W_{p}(\mathcal{L}_{(X,Y)|Z=z},\mathcal{L}_{(X,Y)|Z=z^{\prime}}))\leq\widehat{L_{XY}}|z_{0}-z_{0}^{\prime}|+(2+2^{1+1/p})\varepsilon\qquad\forall z\in V_{j},z^{\prime}\in V_{j^{\prime}}.
\]
Now, note that 
\[
|z_{0}-z|,|z_{0}^{\prime}-z^{\prime}|<\varepsilon/L_{XY}
\]
from the diameter estimate on the bins. Therefore, (still provided
that $\kappa_{q,p,d}M_{q}^{1/q}n_{j}^{-1/d}\leq\varepsilon$, and
with probability at least $1-\exp\left(-\frac{n_{j}\varepsilon^{2p}}{D^{2p}}\right)-\exp\left(-\frac{n_{j^{\prime}}\varepsilon^{2p}}{D^{2p}}\right)$)
\[
W_{p}(\mathcal{L}_{(X,Y)|Z=z},\mathcal{L}_{(X,Y)|Z=z^{\prime}})\leq\widehat{L_{XY}}|z-z^{\prime}|+\left(2+2^{1+1/p}+2\frac{\widehat{L_{XY}}}{L_{XY}}\right)\varepsilon\qquad\forall z\in V_{j},z^{\prime}\in V_{j^{\prime}}.
\]
Finally, quantifying over all bins, we conclude that for all $z$
and $z^{\prime}$ in the support of $\mathcal{L}_{Z}$, 
\[
W_{p}(\mathcal{L}_{(X,Y)|Z=z},\mathcal{L}_{(X,Y)|Z=z^{\prime}})\leq\widehat{L_{XY}}|z-z^{\prime}|+\left(2+2^{1+1/p}+2\frac{\widehat{L_{XY}}}{L_{XY}}\right)\varepsilon
\]
again with probability at least 
\[
1-\underset{j\neq j^{\prime}}{\sum_{j,j^{\prime}=1}^{J}}\left(\exp\left(-\frac{n_{j}\varepsilon^{2p}}{D^{2p}}\right)+\exp\left(-\frac{n_{j^{\prime}}\varepsilon^{2p}}{D^{2p}}\right)\right)\geq1-C\varepsilon^{2d}\exp\left(-\frac{n_{min}\varepsilon^{2p}}{D^{2p}}\right).
\]
 
\end{proof}

\section{Discussion}

It should be emphasized that the class of joint distributions $(X,Y,Z)$
for which a conditional independence test, constructed in the manner
of Theorem \ref{thm:metatheorem}, is feasible, is quite general.
As we have already discussed, the $W_{p}$-Lipschitz continuity we
require is a weaker condition than the $TV$-Lipschitz continuity
assumption from \cite{neykov2020minimax}; nor do we assume that any
of the distributions involved have density with respect to the Lebesgue
measure; nor do we assume that the joint distribution is produced
by any sort of parametric model. Nonetheless, we mention several possible
extensions of the work in this article, stated in approximate order
of increasing difficulty.
\begin{enumerate}
\item \emph{Relaxing the Lipschitz continuity assumption}. One might entertain
other quantitative smoothness conditions on the maps $z\mapsto\mathcal{L}_{(X,Y)|Z=z}$,
$z\mapsto\mathcal{L}_{X|Z=z}$, and $z\mapsto\mathcal{L}_{Y|Z=z}$,
besides Lipschitz continuity, such as Hölder continuity. Indeed, if
one takes the Hölder exponent and constant as given, the proofs in
Section \ref{sec:Conditional-independence-testing} still go through
with only minor adjustments to the epsilon management. Likewise, if
one takes the Hölder exponent as given, it is straightforward to adapt
the arguments in Section \ref{sec:Plug-in-estimation} so as to produce
a consistent estimator for the Hölder constant. 

In principle, more general quantitative smoothness assumptions (such
as an explicitly given modulus of continuity) are also feasible; what
our analysis ultimately demands is \emph{some} sort of quantitative
information about the continuity of the maps $z\mapsto\mathcal{L}_{(X,Y)|Z=z}$,
$z\mapsto\mathcal{L}_{X|Z=z}$, and $z\mapsto\mathcal{L}_{Y|Z=z}$,
so that we can appropriately fix the diameter of the bins.
\item \emph{Allowing the support of $\mathcal{L}_{Z}$ to be unbounded}.
The restriction of the results in Section \ref{sec:Conditional-independence-testing}
to the case where $\mathcal{L}_{Z}$ has bounded support inside $\mathbb{R}^{d_{Z}}$
excludes many natural situations (take for instance $Z$ to be a Gaussian
random variable!). We propose two possible ways to relax this assumption
and work with non-compactly supported $\mathcal{L}_{Z}$, one of which
amounts to ``moving the goalposts'' and the other of which amounts
to ``adding a much stronger assumption somewhere else''. 

For the first, observe that since $\mathcal{L}_{Z}$ is automatically
tight, for every $\varepsilon>0$ there exists a compact set $K_{\varepsilon}$
such that $\mathcal{L}_{Z}(\mathbb{R}^{d_{Z}}\backslash K_{\varepsilon})<\varepsilon$.
Therefore, while it is not possible to form a measurable partition
of $supp(\mathcal{L}_{Z})$ with a finite number of bins and where
the bins all have uniformly diameter (as is required by Theorem \ref{thm:metatheorem}),
what \emph{is} possible is to form such a partition on $K_{\varepsilon}$.
Then, Theorem \ref{thm:metatheorem} indicates how to test the conditional
independence of $X$ and $Y$ given $Z$, \emph{conditional on} $Z\in K_{\varepsilon}$.
This does \emph{not} allow us to test whether $\mathcal{L}_{(X,Y)|Z=z}$
is equal to $\mathcal{L}_{X|Z=z}\otimes\mathcal{L}_{Y|Z=z}$ for \emph{every}
nonzero $z$, but rather for all $z$ lying in a set $K_{\varepsilon}$
where $\mathbb{P}(Z\in K_{\varepsilon})\geq1-\varepsilon$. Call this
``$(1-\varepsilon)$-conditional independence testing''; this discussion
shows that, \emph{if} we are able to come up with an explicit such
$K_{\varepsilon}$for $Z$, and $X$, $Y$, and $Z$ satisfy the other
assumptions of Theorem \ref{thm:metatheorem}, then $(1-\varepsilon)$-conditional
independence testing is feasible. 

As for the second: suppose that, instead of $z\mapsto\mathcal{L}_{(X,Y)|Z=z}$
and $z\mapsto\mathcal{L}_{X|Z=z}$ and $z\mapsto\mathcal{L}_{Y|Z=z}$
merely being $W_{p}$-Lipschitz maps, we require that these maps moreover
have a local Lipschitz constant which quickly goes to zero as $|z|\rightarrow\infty$.
In this situation, it becomes permissible for the bins partitioning
$\mathcal{L}_{Z}$ to have larger and larger diameter as $|z|\rightarrow\infty$;
for bins $z\in V_{j}$ where all $z$ are ``very far'' from $0$,
it is even permissible for the bin to be unbounded towards infinity,
and still have small discretization error. We are not aware of an
application area in conditional independence testing where such an
``asymptotically very $W_{p}$ smooth'' assumption would be natural,
but from from an \emph{analytical} standpoint such an assumption is
sufficient for $W_{p}$ conditional independence testing to be feasible
with unbounded $\mathcal{L}_{Z}$.
\item \emph{Data-dependent bins}. In Section \ref{sec:Conditional-independence-testing},
we assumed that the support of $\mathcal{L}_{Z}$ was first partitioned
into bins, and then samples are drawn from the joint distribution
$(X,Y,Z)$. In this setting, Theorem \ref{thm:metatheorem} then indicates
how many data points must lie in each bin in order to have global
control of the Type I \& II error for the conditional independence
test. Or (what is much the same thing), the bins are produced \emph{independently
}from the data.

In a target application where data is plentiful and the bottleneck
is the computational cost of the empirical Wasserstein distance (see
discussion in Section \ref{sec:Two-sample-independence-testing}),
this is not too concerning. However, in general it might be desirable
to first collect ``as many samples as one can'', and \emph{then}
partition the space where $\mathcal{L}_{Z}$ resides in such a way
that each bin has ``enough'' samples, and then seek various finite
sample guarantees. In this situation, the set of bins (equivalently,
the measurable partition induced by the bins) becomes a random variable
which depends on the samples. This \emph{significantly} complicates
much of the analysis --- see, for instance, the work Canonne et al.
\cite{canonne2018testing}, which offers a conditional independence
test for random variables which take values in a \emph{finite set},
and which \emph{does} offer a binning scheme where the bins are allocated
in a data-dependent fashion. In our (continuous) setting, we leave
such analysis to future work.
\item \emph{Random variables taking values in spaces other than $\mathbb{R}^{d_{X}}\times\mathbb{R}^{d_{Y}}\times\mathbb{R}^{d_{Z}}$.}
We have stipulated that the random variables $X$, $Y$, and $Z$
are Euclidean-valued for concreteness. However, the general theory
of Wasserstein spaces/optimal transport can be carried out in the
rather general setting of complete separable metric spaces \cite{ambrosio2013user,villani2003topics},
and some of the statistical results in, for instance, \cite{boissard2014mean,weed2019sharp}
are stated for more general metric spaces. Most tantalizing is the
recent result of Lei \cite{lei2020convergence}, which offers an expected
error bound, similar to the one of Dereich et al. from \cite{dereich2013constructive},
in an infinite-dimensional setting. (However, \cite{lei2020convergence}
does not provide explicit constants, which our statistical applications
would certainly require.) 

Unfortunately, in (say) an infinite-dimensional Banach space one runs
into the difficulty that (for say $\varepsilon<\frac{1}{2}$) the
$\varepsilon$-covering number of the unit ball is infinite, so even
if $\mathcal{L}_{Z}$ has bounded support, binning the support of
$\mathcal{L}_{Z}$ seems challenging! (On the other hand, if we assume
that $\mathcal{L}_{Z}$ has \emph{compact} support, then this essentially
means that $\mathcal{L}_{Z}$ resides in a finite-dimensional subspace,
and the putative infinite-dimensional setting trivializes.) On possible
strategy is to exploit the \emph{tightness} of (Radon) probability
measures --- that is, even in an infinite dimensional space, if $\mathcal{L}_{Z}$
is Radon we know that for every $\varepsilon>0$, there exists a compact
$K_{\varepsilon}$ with $\mathcal{L}_{Z}(K_{\varepsilon})>1-\varepsilon$.
If, moreover, modeling assumptions tell us that $Z$ must e.g. be
concentrated in some fashion that actually \emph{tells us} a compact
set which serves as a $K_{\varepsilon}$, it is conceivable that,
with a suitable infinite-dimensional concentration inequality, and
a more explicit error estimate resembling the one from \cite{lei2020convergence},
that one could perform ``$(1-\varepsilon$)-conditional independence
testing'' as per (2). However we expect that the requisite analysis
to construct such a test would be challenging.
\end{enumerate}

\section*{Acknowledgments}

The author thanks Tudor Manole, Dejan Slep\v{c}ev, Larry Wasserman,
and Kun Zhang for helpful discussions.

\bibliographystyle{amsalpha}
\bibliography{0_home_andrew_Documents_Andrew_Backup_October_2022_andrew_Documents_otrefs}

\appendix

\section{Proof of Lemma \ref{lem:disintegration}\label{Proof of disintegration lemma}}

Before proceeding, we require the following technical measure-theoretic
result, which we will use to handle the potential non-uniqueness of
optimal transport plans in our setting. 
\begin{thm*}[{Arsenin-Kunugui uniformization theorem \cite[Theorem 35.46 (ii)]{kechris2012classical}}]
 Let $X$ and $Y$ be Polish spaces. Let $R\subseteq X\times Y$
be a Borel set. Suppose that for every $x\in X$, $\{y\in Y:(x,y)\in R\}$
is a countable union of compact sets. Then the projection of $R$
onto $X$ is a Borel subset of $X$, and there exists a \emph{uniformization
}of \emph{$R$,} that is, a Borel function $f:\text{proj}_{X}R\rightarrow Y$
such that $\{x,f(x)\}\in R$ for all $x\in\text{proj}_{X}R$. 
\end{thm*}

\begin{proof}[Proof of Lemma \ref{lem:disintegration}]
 %
Let $V\subseteq Z$ where $Z$ is some Polish space.

Given any $z\in Z$, define $Opt(\mu_{z},\nu)$ to be the set of optimal
transport plans between the probability measures $\mu_{z}$ and $\nu$.
Since $\mathcal{P}(U)$ is a Polish space (when equipped with the
narrow topology), it is known that for fixed $z$, $Opt(\mu_{z},\nu)$
is compact inside the Polish space $\mathcal{P}(U\times U)$ \cite[Corr. 5.21]{villani2008optimal}.
Also this set is always non-empty \cite[Theorem 1.7]{santambrogio2015optimal}.
Moreover, it follows from \cite[Theorem 5.20]{villani2008optimal}
that the set 
\[
Opt(\cdot,\nu):=\{\gamma\in\mathcal{P}(U\times U):\exists\mu\in\mathcal{P}(U),\pi\in Opt(\mu,\nu)\}
\]
 is closed inside $\mathcal{P}(U\times U)$, and is therefore itself
a Polish space. 

Now, let consider the projection map 
\[
\text{proj}_{1}:Opt(\cdot,\nu)\rightarrow\mathcal{P}(U)
\]
which, given a transport plan $\gamma\in Opt(\mu,\nu)$, returns the
first marginal of $\gamma$, namely $\mu$. It is not hard to see
that $\Pi_{1}$ is narrowly continuous: indeed, consider some bounded
continuous function $\varphi$ on $U^{2}$ that only depends on the
first coordinate; in this case, 
\[
\int_{U^{2}}\varphi\gamma=\int_{U}\varphi\text{proj}_{1}\gamma
\]
so in particular, for any narrowly convergent sequence $\gamma_{n}$
of transport plans (with limit $\gamma$), we have 
\[
\int\varphi\gamma_{n}=\int\varphi\text{proj}_{1}\gamma_{n};\quad\int\varphi\gamma=\int\varphi\text{proj}_{1}\gamma
\]
which implies that $\int\varphi\text{proj}_{1}\gamma_{n}\rightarrow\int\varphi\text{proj}_{1}\gamma$
for any bounded continuous $\varphi$ on $U^{2}$, as desired.

Since $\Pi_{1}$ is narrowly continuous, it follows that the transpose
of the graph of $\text{proj}_{1}$, namely the set 
\[
\text{Gr}^{T}\text{proj}_{1}:=\{(\mu,\gamma)\in\mathcal{P}(U)\times Opt(\cdot,\nu):\mu=\text{proj}_{1}\gamma\},
\]
is closed (when both $Opt(\cdot,\nu)$ and $\mathcal{P}(U)$ are equipped
with the narrow topology), and so, in particular, is Borel. Since
both $Opt(\cdot,\nu)$ and $\mathcal{P}(U)$ are Polish, it follows
from the Arsenin-Kunugui uniformization theorem that there exists
a Borel function $OptSelect$ from $\text{proj}_{1}Opt(\cdot,\nu)$
to $Opt(\cdot,\nu)$, whose graph is contained within $\text{Gr}^{T}\text{proj}_{1}$.
But since $\text{proj}_{1}Opt(\cdot,\nu)=\mathcal{P}(U)$ (since $Opt(\mu,\nu)$
is non-empty for every $\mu,\nu\in\mathcal{P}(U)$) it follows that
the domain of $OptSelect$ is all of $\mathcal{P}(U)$. In other words,
$OptSelect(\mu)$ is a \emph{Borel measurable selection of a transport
plan in $Opt(\mu,\nu)$}.

Now, we have assumed that the function $z\mapsto\mu_{z}$ is Borel,
so %
{} by composing $z\mapsto\mu_{z}$ with $OptSelect,$we get a Borel
measurable mapping 
\[
z\mapsto\gamma_{z}:\quad\text{proj}_{1}\gamma_{z}=\mu_{z}.
\]

We also make the following measure-theoretic observation. The narrow
topology on probability measures has the property that for every open
set $O$, the evaluation map $\mu\mapsto\mu(O)$ is continuous; in
particular, this implies (by the monotone class theorem; see also
\cite[Prop. 7.25]{bertsekasstochastic}) that for every Borel set
$B$, the evaluation map $\mu\mapsto\mu(B)$ is Borel. Consequently,
for our measurable selection $z\mapsto\gamma_{z}$, it holds simultaneously
that (1) for every $z\in Z$, $\gamma_{z}$ is a probability measure
(on the space $U\times U$), and that (2) for fixed Borel set $B$,
the map $z\mapsto\gamma_{z}(B)$ is Borel (since it is the composition
of Borel maps). In other words, our measurable selection $\gamma_{z}$
is a (Borel measurable) \emph{stochastic kernel}. 

We claim that for any Borel measurable $V\subseteq Z$, the measure
\[
\int_{V}\gamma_{z}d\lambda(z)
\]
is a transport plan with first marginal $\int_{V}\mu_{z}d\lambda(z)$.
Indeed, take any measurable set $B\subseteq U$; then, since $\gamma_{z}$
is a stochastic kernel, we can use the following elementary property
of stochastic kernels (%
see e.g. \cite[Thm. I.6.3]{ccinlar2011probability}):
\begin{align*}
\int_{U\times U}1_{A\times U}d\left(\int_{V}\gamma_{z}d\lambda(z)\right) & =\int_{V}\left(\int_{U\times U}1_{A\times U}d\gamma_{z}\right)d\lambda(z)\\
 & =\int_{V}\left(\int_{U}1_{A}d\mu_{z}\right)d\lambda(z)\\
 & =\int_{U}1_{A}d\left(\int_{V}\mu_{z}d\lambda(z)\right)
\end{align*}
and so $\left(\int_{V}\gamma_{z}d\lambda(z)\right)(A\times U)=\left(\int_{V}\mu_{z}d\lambda(z)\right)(A)$.
Moreover, by the same reasoning, we see that the second marginal is
$\nu$.

Now, compute as follows. Since $\int_{V}\gamma_{z}d\lambda(z)$ is
a transport plan between $\int_{V}\mu_{z}d\lambda(z)$ and $\nu$,
it holds that 
\[
W_{p}^{p}\left(\int_{V}\mu_{z}d\lambda(z),\nu\right)\leq\int_{U^{2}}d(u_{1},u_{2})^{p}d\left(\int_{V}\gamma_{z}d\lambda(z)\right)(u_{1},u_{2}).
\]
Using \cite[Thm. I.6.3]{ccinlar2011probability} again, we deduce
that
\[
\int_{U^{2}}d(u_{1},u_{2})^{p}d\left(\int_{V}\gamma_{z}d\lambda(z)\right)(u_{1},u_{2})=\int_{V}\left(\int_{U^{2}}d(u_{1},u_{2})^{p}d\gamma_{z}(u_{1},u_{2})\right)d\lambda(z).
\]
But note that 
\[
\int_{U^{2}}d(u_{1},u_{2})^{p}d\gamma_{z}(u_{1},u_{2})=W_{p}^{p}(\mu_{z},\nu).
\]
Therefore, 
\begin{align*}
\int_{V}\left(\int_{U^{2}}d(u_{1},u_{2})^{p}d\gamma_{z}(u_{1},u_{2})\right)d\lambda(z) & =\int_{V}W_{p}^{p}(\mu_{z},\nu)d\lambda(z)
\end{align*}
and so
\[
W_{p}^{p}\left(\int_{V}\mu_{z}d\lambda(z),\nu\right)\leq\int_{V}W_{p}^{p}(\mu_{z},\nu)d\lambda(z)
\]
as desired. This completes the proof.
\end{proof}

\section{Additional Results\label{sec:extras}}

\subsection{Concentration of $W_{p}^{p}(\mu_{n},\mu)$}

We first prove Proposition \ref{prop:concentration} in the case where
the support of $\mu$ has general diameter (as opposed to a diameter
of 1, as in \cite{weed2019sharp}). We emphasize that this is extremely
close to what is already done therein; consider this proof more of
a sanity check. 
\begin{proof}[Proof of Proposition \ref{prop:concentration}]
 Let $c(x,y)=d(x,y)^{p}$. Kantorovich duality tells us that 
\[
W_{p}^{p}(\mu_{n},\mu)=\max_{f\in C_{b}}\left[\int fd\mu_{n}-\int f^{c}d\mu\right]
\]
and
\[
f(x)=\inf_{y}\left[f^{c}(y)+d(x,y)^{p}\right]
\]
and 
\[
f^{c}(y)=\sup_{x}\left[f(x)-d(x,y)^{p}\right].
\]
Chose $f\in C_{b}$ so that it attains the maximum for $\int fd\mu_{n}-\int f^{c}d\mu$.
$f$ is defined only up to a constant, so without loss of generality
take $\sup_{x}f(x)=D^{p}$. Consequently, 
\[
f^{c}(y)\geq\sup_{x}\left[D^{p}-d(x,y)^{p}\right]\geq0
\]
and therefore 
\[
f(x)\geq\inf_{y}\left[0+d(x,y)^{p}\right]\geq0.
\]
Thus, we can take the $f\in C_{b}$ which attains the maximum for
$\int fd\mu_{n}-\int f^{c}d\mu$ to have the property that $0\leq f(x)\leq D^{p}$
for all $x$ in the domain. 

Now, let $X_{1},\ldots,X_{n}$ be i.i.d. draws from $\mu$, so that
$\mu_{n}=\frac{1}{n}\sum_{i=1}^{n}X_{i}$. Define the random function
\begin{align*}
w(X_{1},\ldots,X_{n}) & :=W_{p}^{p}(\mu_{n},\mu)\\
 & =\sup_{f\in C_{b}}\left[\int fd\mu_{n}-\int f^{c}d\mu\right]\\
 & =\underset{0\leq f\leq D^{p}}{\sup_{f\in C_{b}}}\left[\frac{1}{n}\sum_{i=1}^{n}f(X_{i})-\int f^{c}d\mu\right].
\end{align*}
Let $j\in\{1,\ldots,n\}$. Observe that given any $x_{1},\ldots,x_{n}$
and $x_{j}^{\prime}$, 
\begin{align*}
w(x_{1},\ldots,x_{j},\ldots,x_{n})-w(x_{1},\ldots,x_{j}^{\prime},\ldots,x_{n}) & \leq\underset{0\leq f\leq D^{p}}{\sup_{f\in C_{b}}}\left[\frac{1}{n}\sum_{i=1}^{n}f(x_{i})-\int f^{c}d\mu\right]\\
 & \quad-\underset{0\leq f^{\prime}\leq D^{p}}{\sup_{f^{\prime}\in C_{b}}}\left[\frac{1}{n}\underset{i\neq j}{\sum_{i=1}^{n}}f^{\prime}(x_{i})+\frac{1}{n}f^{\prime}(x_{j}^{\prime})-\int\left(f^{\prime}\right)^{c}d\mu\right].
\end{align*}
Let 
\[
\tilde{f}\in\underset{0\leq f\leq D^{p}}{\underset{f\in C_{b}}{\text{argmax}}}\left[\frac{1}{n}\sum_{i=1}^{n}f(x_{i})-\int f^{c}d\mu\right].
\]
Note that 
\[
\left[\frac{1}{n}\underset{i\neq j}{\sum_{i=1}^{n}}\tilde{f}(x_{i})+\frac{1}{n}\tilde{f}(x_{j}^{\prime})-\int\left(\tilde{f}\right)^{c}d\mu\right]\leq\underset{0\leq f^{\prime}\leq D^{p}}{\sup_{f^{\prime}\in C_{b}}}\left[\frac{1}{n}\underset{i\neq j}{\sum_{i=1}^{n}}f^{\prime}(x_{i})+\frac{1}{n}f^{\prime}(x_{j}^{\prime})-\int\left(f^{\prime}\right)^{c}d\mu\right]
\]
and thus 
\begin{align*}
 & \underset{0\leq f\leq D^{p}}{\sup_{f\in C_{b}}}\left[\frac{1}{n}\sum_{i=1}^{n}f(x_{i})-\int f^{c}d\mu\right]-\underset{0\leq f^{\prime}\leq D^{p}}{\sup_{f^{\prime}\in C_{b}}}\left[\frac{1}{n}\underset{i\neq j}{\sum_{i=1}^{n}}f^{\prime}(x_{i})+\frac{1}{n}f^{\prime}(x_{j}^{\prime})-\int\left(f^{\prime}\right)^{c}d\mu\right]\\
 & \leq\left[\frac{1}{n}\sum_{i=1}^{n}\tilde{f}(x_{i})-\int\tilde{f}^{c}d\mu\right]-\left[\frac{1}{n}\underset{i\neq j}{\sum_{i=1}^{n}}\tilde{f}(x_{i})+\frac{1}{n}\tilde{f}(x_{j}^{\prime})-\int\left(\tilde{f}\right)^{c}d\mu\right]\\
 & \leq\frac{1}{n}\left[\tilde{f}(x_{j})-\tilde{f}(x_{j}^{\prime})\right]\\
 & \leq\frac{D^{p}}{n}.
\end{align*}
It therefore follows from McDiarmid's inequality that 
\[
\mathbb{P}\left[W_{p}^{p}(\mu_{n},\mu)\geq\mathbb{E}W_{p}^{p}(\mu_{n},\mu)+t\right]\leq\exp\left(-\frac{2nt^{2}}{D^{2p}}\right)
\]
as desired.
\end{proof}
Let us now adapt the previous proof to the case where $\mu$ is not
necessarily bounded, using Combes's variant on McDiarmid's inequality,
proved in \cite{combes2015extension}. Mostly we are interested in
a concentration inequality under the same assumptions as in Theorem
\ref{thm:expectation bound}, namely that $d\geq3$, $p\in[1,d/2)$,
and $\mu$ has a $q$th moment bound for $q$ such that $q>dp/(d-p)$.

Let's recall the statement of Combes's variant on McDiarmid's inequality.
\begin{thm*}[Combes \cite{combes2015extension}]
 Let $(\mathcal{X}_{1},\ldots,\mathcal{X}_{n})$ be sets and define
$\mathcal{X}:=\prod_{i=1}^{n}\mathcal{X}_{i}$. Let $X:=(X_{1},\ldots,X_{n})$
be independent random variables with $X_{i}\in\mathcal{X}_{i}$. Let
$w:\mathcal{X}\rightarrow\mathbb{R}$, so that $w(X)$ is a random
variable. 

Let $\mathcal{Y}\subset\mathcal{X}$. Define $\mathfrak{p}:=1-\mathbb{P}[X\in\mathcal{Y}]$
and $m=\mathbb{E}[w(X)\mid X\in\mathcal{Y}]$. Let $\vec{c}:=(c_{1},\ldots,c_{n})\in\mathbb{R}_{+}^{n}$
and define $\bar{c}=\sum_{i=1}^{n}c_{i}$. Suppose that: for all $j\in\{1,\ldots,n\}$,
for all $(y,y^{\prime})\in\mathcal{Y}^{2}$ with $y_{i}=y_{i}^{\prime}$
for all $i\neq j$, it holds that $|w(y)-w(y^{\prime})|\leq c_{j}$.
Then, for all $t>0$, it holds that 
\[
\mathbb{P}[w(X)\geq m+t]\leq\mathfrak{p}+\exp\left(-\frac{2(t-\mathfrak{p}\bar{c})_{+}^{2}}{\sum_{i=1}^{n}c_{i}^{2}}\right).
\]
\end{thm*}
We will apply Combes's inequality in the case where: $(X_{1},\ldots,X_{n})$
are i.i.d. samples from some probability measure $\mu$ (and we denote
$\mu_{n}=\frac{1}{n}\sum_{i=1}^{n}X_{i}$), $\mathcal{Y}=B(e,R)^{n}$
for some ``large'' radius $R$ and arbitrary ``origin'' $e\in X$,
and $w(X)=W_{p}^{p}(\mu_{n},\mu)$. In order to apply Combes's inequality,
we must compute $c_{i}$, $m$, and $\mathfrak{p}$. We do so in the
situation where $\mu$ has bounded $q$th moment, where $q>p$. 
\begin{prop}[$W_{p}$ empirical measure concentration inequality with higher moment
bound]
 \label{prop:unbounded Wp concentration ineq} Let $(X,d)$ be a
Polish metric space with (arbitrary) distinguished point $e\in X$.
Let $\mu\in\mathcal{P}(X)$ with $\left(\int_{X}d(e,x)^{q}d\mu(x)\right)^{1/q}\leq M_{q}$
for some $q>p\geq1$. Let $\mu_{n}=\sum_{i=1}^{n}\delta_{x_{i}}$
be a (random) empirical measure for $\mu$. Furthermore, let $\eta\in(0,1)$,
and let
\[
\bar{\mu}\upharpoonright B(e,M_{q}\eta^{-1/q}):=\frac{1}{\mu(B(e,M_{q}\eta^{-1/q}))}\mu\upharpoonright B(e,M_{q}\eta^{-1/q})
\]
denote the normalized restriction of $\mu$ to $B(e,M_{q}\eta^{-1/q})$,
and let $\mu_{n}\mid(X_{1},\ldots,X_{n})\in B(e,M_{q}\eta^{-1/q})^{n}$
denote the conditional empirical measure where $x_{i}\in B(e,M_{q}\eta^{-1/q})^{n}$
for all $i=1,\ldots,n$. Then, for any $t>0$, it holds that 
\begin{multline*}
\mathbb{P}\left[W_{p}^{p}(\mu_{n},\mu)\geq\mathbb{E}\left[\left(W_{p}\left(\mu_{n}\mid(X_{1},\ldots,X_{n})\in B(e,M_{q}\eta^{-1/q})^{n},\bar{\mu}\upharpoonright B(e,M_{q}\eta^{-1/q})\right)+c_{p,q}M_{q}\eta^{1/p-1/q}\right)^{p}\right]+t\right]\\
\leq1-(1-\eta)^{n}+\exp\left(-\frac{2n\left(t-(1-(1-\eta)^{n})((M_{q}\eta^{-1/q})^{p})\right)_{+}^{2}}{(M_{q}\eta^{-1/q})^{2p}}\right)
\end{multline*}
where $c_{p,q}:=2^{1/q}\frac{2^{1/p-1/q}}{2^{1/p-1/q}-1}+\frac{2^{1/p}}{2^{1/p}-1}$.
\end{prop}

\begin{proof}
First, recall that every (Radon) probability measure $\mu$ on a Polish
space is automatically tight --- that is, for every $\eta\in(0,1)$
there exists a compact $K_{\varepsilon}$ such that $\mu(K_{\eta})>1-\eta$.
If, moreover, $\mu$ has bounded $q$th moment, we can say more. Indeed,
suppose that $\mu(X\backslash B(e,R))\geq\eta$. Then, 
\[
R\eta^{1/q}\leq\left(\int_{X}d(e,x)^{q}d\mu(x)\right)^{1/q}.
\]
It follows that if we know that $\left(\int_{X}d(e,x)^{q}d\mu(x)\right)^{1/q}\leq M_{q}$,
then $R\eta^{1/q}\leq M_{q}$, in other words, the maximal radius
at which $\mu(X\backslash B(e,R))\geq\eta$ is $M_{q}\eta^{-1/q}$.
Likewise, if we have $n$ i.i.d. draws from $\mu$, we know that for
each draw, the probability of landing outside of $B(e,M_{q}\eta^{-1/q})$
is at most $\eta$. %
It follows that the probability of all the draws landing inside $B(e,M_{q}\eta^{-1/q})$
--- which is the quantity $1-\mathfrak{p}$ in the statement of Combes's
inequality --- is at least $(1-\eta)^{n}$. Hence $\mathfrak{p}\leq1-(1-\eta)^{n}$. 

To compute $m$ from the statement of Combes's inequality, we will
need to upper bound $W_{p}^{p}(\mu,\bar{\mu}\upharpoonright B(e,M_{q}\eta^{-1/q}))$
in the case where $p<q$. So, compute as follows. We have seen that
$\mu(B(e,M_{q}\left(\frac{\eta}{2^{i}}\right)^{-1/q}))\geq1-\frac{\eta}{2^{i}}$.
It follows that 
\[
\mu\left(B(e,M_{q}\left(\frac{\eta}{2^{i}}\right)^{-1/q})\backslash B(e,M_{q}\left(\frac{\eta}{2^{i-1}}\right)^{-1/q})\right)\leq\mu\left(\mathbb{R}^{d}\backslash B(e,M_{q}\left(\frac{\eta}{2^{i-1}}\right)^{-1/q})\right)<\frac{\eta}{2^{i-1}}.
\]
At the same time, regardless of the choice of transport plan between
$\bar{\mu}\upharpoonright B(e,M_{q}\eta^{-1/q})$ and $\mu$, when
transporting the mass in $\mu$ which resides in $B(e,M_{q}\left(\frac{\eta}{2^{i}}\right)^{-1/q})\backslash B(e,M_{q}\left(\frac{\eta}{2^{i-1}}\right)^{-1/q}$
onto $\bar{\mu}\upharpoonright B(e,M_{q}\eta^{-1/q})$, the distance
traveled by this mass is at most $M_{q}\left(\frac{\eta}{2^{i}}\right)^{-1/q}+M_{q}\eta^{-1/q}$
(for each particle of mass), and the total mass transported from $B(e,M_{q}\left(\frac{\eta}{2^{i}}\right)^{-1/q})\backslash B(e,M_{q}\left(\frac{\eta}{2^{i-1}}\right)^{-1/q})$
is less than $\frac{\eta}{2^{i-1}}$. Consequently, the $W_{p}$ transportation
cost to transport this portion of the mass in $\mu$ is less than
$\left(M_{q}\left(\frac{\eta}{2^{i}}\right)^{-1/q}+M_{q}\eta^{-1/q}\right)\left(\frac{\eta}{2^{i-1}}\right)^{1/p}$. 

Furthermore, since the mass in $\frac{1}{\mu(B(e,M_{q}\eta^{-1/q}))}\mu\upharpoonright B(e,M_{q}\eta^{-1/q})$
is everywhere greater than the mass in $\mu\upharpoonright B(e,M_{q}\eta^{-1/q})$,
it follows that we can choose a transport plan from $\mu$ to $\frac{1}{\mu(B(e,M_{q}\eta^{-1/q}))}\mu\upharpoonright B(e,M_{q}\eta^{-1/q})$
will leave all of $\mu$'s mass within $B(e,M_{q}\eta^{-1/q})$ unmoved.
Therefore, we can upper bound 
\begin{align*}
W_{p}(\mu,\frac{1}{\mu(B(e,M_{q}\eta^{-1/q}))}\mu\upharpoonright B(e,M_{q}\eta^{-1/q})) & <\sum_{i=1}^{\infty}\left(M_{q}\left(\frac{\eta}{2^{i}}\right)^{-1/q}+M_{q}\eta^{-1/q}\right)\left(\frac{\eta}{2^{i-1}}\right)^{1/p}\\
 & =M_{q}\eta^{1/p-1/q}\sum_{i=1}^{\infty}\left(\frac{1}{(2^{i})^{-1/q}(2^{i-1})^{1/p}}+\frac{1}{(2^{i-1})^{1/p}}\right)\\
 & =M_{q}\eta^{1/p-1/q}\sum_{i=1}^{\infty}\left(\frac{2^{1/q}}{(2^{i-1})^{-1/q}(2^{i-1})^{1/p}}+\frac{1}{(2^{i-1})^{1/p}}\right)\\
 & =M_{q}\eta^{1/p-1/q}\left(2^{1/q}\frac{2^{1/p-1/q}}{2^{1/p-1/q}-1}+\frac{2^{1/p}}{2^{1/p}-1}\right).
\end{align*}
Letting $c_{p,q}:=2^{1/q}\frac{2^{1/p-1/q}}{2^{1/p-1/q}-1}+\frac{2^{1/p}}{2^{1/p}-1}$,
we see that 
\[
W_{p}(\mu,\frac{1}{\mu(B(e,M_{q}\eta^{-1/q}))}\mu\upharpoonright B(e,M_{q}\eta^{-1/q}))<c_{p,q}M_{q}\eta^{1/p-1/q}.
\]
Note that $c_{p,q}\rightarrow\infty$ as $q\searrow p$. 

We are now in a position to compute $c_{i}$ and $m$ (again from
the statement of Combes's inequality). First, note that the random
measure $\mu_{n}$, but conditioned on the event $(X_{1},\ldots,X_{n})=X\in B(e,M_{q}\eta^{-1/q})^{n}$,
is an empirical measure for $\bar{\mu}\upharpoonright B(e,M_{q}\eta^{-1/q})$.
Thus, reasoning as in the proof of Proposition \ref{prop:concentration},
we see that 
\begin{multline*}
W_{p}^{p}\left(\mu_{n}\mid(X_{1},\ldots,X_{n})\in B(e,M_{q}\eta^{-1/q})^{n},\bar{\mu}\upharpoonright B(e,M_{q}\eta^{-1/q})\right)=\underset{0\leq f\leq(M_{q}\eta^{-1/q})^{p}}{\sup_{f\in C_{b}}}\left[\frac{1}{n}\sum_{i=1}^{n}f(X_{i})-\int f^{c}d\mu\right]
\end{multline*}
and likewise, conditioned on the event $X\in B(e,M_{q}\eta^{-1/q})^{n}$,
the proof of Proposition \ref{prop:concentration} shows that for
all $j\in\{1,\ldots,n\}$, 
\[
w(x_{1},\ldots,x_{j},\ldots,x_{n})-w(x_{1},\ldots,x_{j}^{\prime},\ldots,x_{n})\leq\frac{(M_{q}\eta^{-1/q})^{p}}{n}:=c_{j}.
\]
Hence $\bar{c}:=\sum_{i=1}^{m}c_{i}=(M_{q}\eta^{-1/q})^{p}$, and
$\sum_{i=1}^{n}c_{i}^{2}=\frac{(M_{q}\eta^{-1/q})^{2p}}{n}$. At the
same time, by the triangle inequality, %
\begin{multline*}
m:=\mathbb{E}W_{p}^{p}\left(\mu_{n}\mid(X_{1},\ldots,X_{n})\in B(0,M_{q}\eta^{-1/q})^{n},\mu\right)\leq\\
\mathbb{E}\left[\left(W_{p}\left(\mu_{n}\mid(X_{1},\ldots,X_{n})\in B(e,M_{q}\eta^{-1/q})^{n},\bar{\mu}\upharpoonright B(e,M_{q}\eta^{-1/q})\right)+W_{p}\left(\mu,\bar{\mu}\upharpoonright B(e,M_{q}\eta^{-1/q})\right)\right)^{p}\right]
\end{multline*}
\[
\leq\mathbb{E}\left[\left(W_{p}\left(\mu_{n}\mid(X_{1},\ldots,X_{n})\in B(e,M_{q}\eta^{-1/q})^{n},\bar{\mu}\upharpoonright B(e,M_{q}\eta^{-1/q})\right)+c_{p,q}M_{q}\eta^{1/p-1/q}\right)^{p}\right].
\]

Therefore, Combes's inequality tells us that: for all $\eta\in(0,1)$
and $t>0$, 
\begin{multline*}
\mathbb{P}\left[W_{p}^{p}(\mu_{n},\mu)\geq\mathbb{E}\left[\left(W_{p}\left(\mu_{n}\mid(X_{1},\ldots,X_{n})\in B(e,M_{q}\eta^{-1/q})^{n},\bar{\mu}\upharpoonright B(e,M_{q}\eta^{-1/q})\right)+c_{p,q}M_{q}\eta^{1/p-1/q}\right)^{p}\right]+t\right]\\
\leq\mathbb{P}\left[W_{p}^{p}(\mu_{n},\mu)\geq m+t\right]\leq1-(1-\eta)^{n}+\exp\left(-\frac{2n\left(t-(1-(1-\eta)^{n})((M_{q}\eta^{-1/q})^{p})\right)_{+}^{2}}{(M_{q}\eta^{-1/q})^{2p}}\right).
\end{multline*}
\end{proof}
It is possible to combine this concentration inequality with the expectation
bound from Theorem \ref{thm:expectation bound} in a straightforward
way:
\begin{cor}
\label{cor:d>=00003D3 concentration + exp. bound}Under the same assumptions
as the preceding theorem, suppose, in addition, that $X=\mathbb{R}^{d}$,
$p\in[1,d/2)$, and $q>dp/(d-p)$. Then, for all $\eta\in(0,1)$ and
$t>0$, 
\begin{multline*}
\mathbb{P}\left[W_{p}^{p}(\mu_{n},\mu)\geq\left(\kappa_{p,q,d}M_{q}(1+\eta)^{1/q}n^{-1/d}+c_{p,q}M_{q}\eta^{1/p-1/q}\right)^{p}+t\right]\\
\leq1-(1-\eta)^{n}+\exp\left(-\frac{2n\left(t-(1-(1-\eta)^{n})((M_{q}\eta^{-1/q})^{p})\right)_{+}^{2}}{(M_{q}\eta^{-1/q})^{2p}}\right).
\end{multline*}
\end{cor}

\begin{proof}
If $\mu$ has moment bound $M_{q}$, then it follows that $\bar{\mu}\upharpoonright B(0,M_{q}\eta^{-1/q})$
has upper moment bound $M_{q}(1+\eta)^{1/q}$ (because in the worst
case, $\mu$ is concentrated near $\partial B(0,M_{q}\eta^{-1/q})$.)
So, we know, using Theorem \ref{thm:expectation bound}, that under
the requisite additional assumptions of Theorem \ref{thm:expectation bound},
namely that $p\in[1,d/2)$ and $q>dp/(d-p)$, it holds that
\[
\mathbb{E}W_{p}\left(\mu_{n}\mid(X_{1},\ldots,X_{n})\in B(0,M_{q}\eta^{-1/q})^{n},\bar{\mu}\upharpoonright B(0,M_{q}\eta^{-1/q})\right)\leq\kappa_{p,q,d}M_{q}(1+\eta)^{1/q}n^{-1/d}
\]
and so it follows, in this circumstance, that that 
\begin{multline*}
\mathbb{P}\left[W_{p}^{p}(\mu_{n},\mu)\geq\left(\kappa_{p,q,d}M_{q}(1+\eta)^{1/q}n^{-1/d}+c_{p,q}M_{q}\eta^{1/p-1/q}\right)^{p}+t\right]\\
\leq1-(1-\eta)^{n}+\exp\left(-\frac{2n\left(t-(1-(1-\eta)^{n})((M_{q}\eta^{-1/q})^{p})\right)_{+}^{2}}{(M_{q}\eta^{-1/q})^{2p}}\right).
\end{multline*}
\end{proof}

\end{document}